\documentclass[12pt, letterpaper]{amsart}

\newcommand{\la}{\langle}
\newcommand{\ra}{\rangle}
\renewcommand{\Re}{\operatorname{Re}}
\renewcommand{\Im}{\operatorname{Im}}

\usepackage{amssymb}
\usepackage{amsthm}
\usepackage{amsxtra}
\usepackage{graphicx}
\usepackage{color}
\usepackage{xcolor}
\usepackage{mathrsfs}

\newtheorem{theorem}{Theorem}

\newtheorem{proposition}[theorem]{Proposition}
\newtheorem{lemma}[theorem]{Lemma}
\newtheorem{corollary}[theorem]{Corollary}

\theoremstyle{remark}
\newtheorem{remark}[theorem]{Remark}

\numberwithin{equation}{section}

\numberwithin{theorem}{section}

\numberwithin{table}{section}

\numberwithin{figure}{section}

\ifx\pdfoutput\undefined
  \DeclareGraphicsExtensions{.pstex, .eps}
\else
  \ifx\pdfoutput\relax
    \DeclareGraphicsExtensions{.pstex, .eps}
  \else
    \ifnum\pdfoutput>0
      \DeclareGraphicsExtensions{.pdf}
    \else
      \DeclareGraphicsExtensions{.pstex, .eps}
    \fi
  \fi
\fi

\title[Effective Dynamics of the Vector NLS on Large Domains]{Effective Dynamics of the Vector Nonlinear Schr\"odinger Equations on Large Domains}

\author{Katherine Zhiyuan Zhang}
\address{Courant Institute of Mathematical Sciences, New York University}

\begin{document}

\maketitle

\begin{abstract}
We consider the vector nonlinear Schr\"odinger equation posed on the box with periodic boundary conditions $\mathbb{T}^n_L$, and derive the continuous resonant (CR) equation that describes the effective dynamics for large $L$ and small data size over very large time-scales. Moreover, we investigate various properties of the continuous resonant equations, including the Hamiltonian structure and the well-posedness, etc..
\end{abstract}

\section{Introduction}




We consider the following Vector Nonlinear Schr\"odinger Equations (VNLS) with small initial data:
\begin{equation} \label{VNLS}
-i\partial_t E + \frac{1}{2 \pi} \Delta E = \Delta^{-1} \nabla \nabla \cdot (|E|^2 E) , \ E|_{t=0} = \epsilon E_0 .
\end{equation}
Here $E = E(x, t) : \mathbb{T}^n_L \times \mathbb{R} \rightarrow \mathbb{C}^n$ is an \emph{irrotational} complex-valued vector field (see \cite{CW1}), and $E_0 = E_0(x) : \mathbb{T}^n_L  \rightarrow \mathbb{C}^n$ is a fixed \emph{irrotational} complex-valued vector field. Writing \eqref{VNLS} in different components yields
\begin{equation}
-i\partial_t E_j + \frac{1}{2 \pi} \Delta E_j = \big[ \Delta^{-1} \nabla \nabla \cdot (|E|^2 E)\big]_j , \ E_j|_{t=0} = \epsilon E_{0, j} , \ j = 1, 2, \cdots , n .
\end{equation}
The equation \eqref{VNLS} is motivated by the mathematical description of nonlinear (Langmuir) waves in a collisionless plasma. In fact, a collisionless plasma can be described by the coupled evolution of the electron plasma wave envelope, denoted by $E(x,t)$, the governing Langmuir waves, and the ion density fluctuations about its equilibrium value, denoted by $\delta n (x, t)$. The evolution equations are known as the Zakharov equations (see, for example, \cite{Dendy1}, \cite{GTWT1}, \cite{ZMR1}) 
\begin{equation} \label{VNLS-plasma}
\begin{split}
& - i \frac{\partial E}{\partial t} = \alpha^2 \nabla \times \nabla \times  E - \nabla (\nabla \cdot E ) + n E ,  \\
& \frac{1}{c_s^2} \frac{\partial^2 n}{\partial t^2} = \Delta (n + |E|^2) , \\
\end{split}
\end{equation}
where $E$ is related to the slowly varying envelope of the electric field $\mathfrak{E}$ by 
\begin{equation*}
\mathfrak{E} = \frac{1}{2} (E e^{-i \omega_p t} + E^* e^{i \omega_p t} ) .
\end{equation*}
The parameter $\alpha^2$ is given by
\begin{equation*}
\alpha^2 = \frac{c^2}{3 v_{T_e}^2} , 
\end{equation*}
where $c$ is the speed of light, and $v_{T_e}$ is the thermal velocity of the electrons, and $\omega_p$ is the plasma frequency (see, for example, \cite{Dendy1}). In the limit $c_s \rightarrow \infty$, when $E(x, t)$ is an \emph{irrotational} field, i.e. $E = \nabla \psi$, the term $  \nabla \times \nabla \times  E $ vanishes and the system \eqref{VNLS-plasma} reduces to the vector nonlinear Schr\"odinger equation \eqref{VNLS}. Notice that we put down the operator $\Delta^{-1} \nabla \nabla \cdot$, which is the projection onto the curl free part, to the nonlinearity, so as to make sure that the irrotationality is preserved along the solution flow. We would like to study the long-time behavior in the large box limit $L \rightarrow + \infty$ for \eqref{VNLS}. 

For various nonlinear dispersive PDEs, the qualitative feature of global solutions has been a long standing topic of interest. It is expected that for generic initial data, solutions will exhibit a behavior that is very different than the linear one, even for small data. An example of such behavior is the phenomenon of \emph{energy cascade}, in which solutions transfer their energy to higher and higher Fourier modes. Study on this behavior has been carried out by inspecting the behavior of high Sobolev norms, see, for example, \cite{CKSTT1}, \cite{Hani1}, \cite{GuaKal1}. 

The intuition for energy cascade is highly motivated by the theory of statistical physics of nonlinear dispersive waves, namely \emph{wave turbulence} or \emph{weak turbulence}. An element of the mathematical theory of wave turbulence is the derivation of a \emph{kinetic wave equation} (the KWE equation) which should describe large time dynamics. The derivation of this \emph{kinetic wave equation} is performed under a randomness assumption, in a limiting sense  where the size of the data $\epsilon \rightarrow 0$ (weak nonlinearity), and the size of the domain $L \rightarrow \infty$ (large box limit, or, up to rescaling, high frequency limit).  

The study of the long-time behavior in the large box limit is also motivated by the need to understand the various regimes and effective dynamics that are featured by dispersive systems on large domains (e.g. water waves equation posed on the ocean surface). It has been a tempting thought for people to find simplified equations posed on $\mathbb{R}^n$ that approximate the real dynamics when $L$ is very large. The first and most intuitive such simplification, is to study the same equation on $\mathbb{R}^n$, that is, the Euclidean approximation. We consider data of size $O(\epsilon)$ and consider the effectiveness of the Euclidean approximation when $\epsilon$ is small and $L \rightarrow +\infty$. It turns out that this approximation has its limitations -- namely, it is only valid in situations and over time-scales for which the solution does not feel the effect of the boundary of the domain. It is crucial to understand the effective time scale of this approximation. There are two important time scales to consider: the nonlinear time scale $T_{nl} \sim \epsilon^{-2p}$ and the Euclidean time scale $T_E \sim L$. Beyond $T_{nl}$ and $T_E$, there exists another time scale $T_R$, which we call the \emph{resonant time scale}. 

It was shown in \cite{BGHS1} that for the nonlinear Schr\"odinger equation, the effective dynamics of the solution can be described by an equation on $\mathbb{R}^n$, called the \emph{continuous resonant equation} (the CR equation), over the time scale $T_R$. Their proof relies on tools from analytic number theory, including a relatively modern version of the Hardy-Littlewood circle method, as well as a normal form transformation. Similar study is carried out in \cite{FGH1}, etc.. On the other hand, in a stochastic setting with randomness assumed in the initial data, the KWE equations are derived to describe the large box and weakly nonlinear limit of NLS, see, for example, \cite{BGHS3}, \cite{CG1}, \cite{CG2}, \cite{DH1}, etc..


Properties of the CR equation are studied in various literature. \cite{BGHS2} carried out an investigation of the structure of the CR equation derived in \cite{BGHS1}, its local well-posedness, and the existence of stationary waves. In \cite{GHTho1}, it was shown that CR can arise in another natural way, as it also corresponds to the resonant cubic Hermite-Schr\"odinger equation (NLS with harmonic trapping). The paper also uncovered that that the basis of special Hermite functions is well suited to its analysis. \cite{GHTho2} carried out a probabilistic study of the CR equation. More study on the structure and analytic and variational properties of CR can be seen in \cite{FGH1}, \cite{Fennell1}, etc.. Moreover, properties of the KWE equations are also of great interest and is investigated in various literature, see, for example, \cite{EV1}, \cite{GITran1}, etc..



In this paper, we consider the CR equation for the vector nonlinear Schr\"odinger equation \eqref{VNLS}. We use the ansatz $E= \epsilon u$ to represent our selection of a small data of size $\sim \epsilon$, then the equation for $u$ is
\begin{equation} \label{VNLS-rescaled}
-i\partial_t u + \frac{1}{2 \pi} \Delta u = \epsilon^2 \Delta^{-1} \nabla \nabla \cdot (|u|^2 u) , \ u|_{t=0} =   E_0 ,
\end{equation}
or,
\begin{equation}
-i\partial_t u_j + \frac{1}{2 \pi} \Delta u_j = \epsilon^2 \big[ \Delta^{-1} \nabla \nabla \cdot (|u|^2 u)\big]_j , \ u_j|_{t=0} = E_{0, j} , \ j = 1, 2, \cdots , n .
\end{equation}
We are going to study the long-time behavior of this small initial data solution of \eqref{VNLS} in the rescaled frame (equation \eqref{VNLS-rescaled}) when taking the large box limit $L \rightarrow \infty$. 

We expand $u$ in its Fourier coefficients:
\begin{equation*}
u(t, x) = \frac{1}{L^n} \sum_{K \in \mathbb{Z}_L^n} u_K (t) \ e(K \cdot x)
\end{equation*}
where $K \in \mathbb{Z}_L^n = (\mathbb{Z}/L)^n$, $e(\alpha) = \exp (2 \pi i \alpha) $. Define
\begin{equation*}
a_K (t) = e^{-|K|^2 t} u_K (t) .
\end{equation*}
The equation satisfied by $a_K (t)$ is
\begin{equation*}
- i \partial_t a_K = \frac{\epsilon^2}{L^{2n}}   \frac{K(j)}{|K|^2} \sum_{m=1}^n K(m) \sum_{S_3 (K)=0  } (\sum_{l=1}^n a_{l, K_1} \overline{a}_{l, K_2} )  a_{m, K_3} \exp (2 \pi i \Omega_3 (K) t ) .
\end{equation*}
We can split the nonlinear terms into resonant and non-resonant interactions:
\begin{equation*}
\begin{split}
- i \partial_t a_K = 
& \frac{\epsilon^2}{L^{2n}}   \frac{K(j)}{|K|^2} \sum_{m=1}^n K(m) \sum_{S_3 (K)=0 , \  \Omega_3 (K) = 0  } (\sum_{l=1}^n a_{l, K_1} \overline{a}_{l, K_2} )  a_{m, K_3}   \\
& + \frac{\epsilon^2}{L^{2n}}   \frac{K(j)}{|K|^2} \sum_{m=1}^n K(m) \sum_{S_3 (K)=0  , \  \Omega_3 (K) \neq 0   } (\sum_{l=1}^n a_{l, K_1} \overline{a}_{l, K_2} )  a_{m, K_3} \exp (2 \pi i \Omega_3 (K) t ) . \\
\end{split}
\end{equation*}
Using a normal form transformation, we can show that for $\epsilon$ sufficiently small, the non-resonant interactions become dynamically irrelevant, and the dynamics of small solutions are well-approximated by the resonant terms. Moreover, for $L$ sufficiently large, we can utilize tools from analytic number theory to approximate the resonant sum by an integral in a manner similar to how Riemann sums are approximated by integrals.


 
It turns out that the limiting dynamics of $a_K(t)$ (up to rescaling time by a factor $\frac{L^{2n}}{Z_n (L) \epsilon^2}$) is given by a \emph{continuous resonant equation} with the effective time scale $\sim T_R$, just as the case of NLS in \cite{BGHS1}. In this paper, we will show, for $n \geq 3$, as $\epsilon \rightarrow 0$, $L \rightarrow \infty$, $\epsilon<L^{-\gamma}$ for some $\gamma >0$, the effective dynamics of the solution of \eqref{VNLS-rescaled} is depicted by the following \emph{continuous resonant equation} (CR):
\begin{equation} \label{CR}
-i\partial_t g(t,K)= \mathcal{T} (g(t, \cdot) , g(t, \cdot) , g(t, \cdot)) (t, K) , \ K \in \mathbb{R}^n .
\end{equation}
For $n=2$, as $\epsilon \rightarrow 0$, $L \rightarrow \infty$, $\epsilon<L^{-\gamma}$ for some $\gamma >0$, the effective dynamics of the solution of \eqref{VNLS-rescaled} is depicted by the following \emph{continuous resonant equation} (CR):
\begin{equation} \label{CR-n=2}
-i\partial_t g(t,K)= \mathcal{T} (g(t, \cdot) , g(t, \cdot) , g(t, \cdot)) (t, K) + \frac{\zeta(2)}{\log L}  \sum_{l=1}^n \frac{K(j)}{|K|^2} \sum_{m=1}^n K(m) \mathcal{C} (g(t, \cdot)) (t, K) , \ K \in \mathbb{R}^2 .
\end{equation}
Here $\mathcal{T} = (\mathcal{T}_j )_{j=1, 2, \cdots , n} $, 
\begin{equation*}
\mathcal{T}_j (g, g, g) (K) := \frac{K(j)}{|K|^2} \sum_{m=1}^n K(m) \int_{\mathbb{R}^{3n}} \big(\sum_{l=1}^n g_l (K_1) \overline{g}_l (K_2) \big) g_m (K_3) \delta_{\mathbb{R}^n} (S_3 (K)) \delta_{\mathbb{R}} (\Omega_3 (K)) dK_1 dK_2 dK_3  , 
\end{equation*}
\begin{equation*}
S_3 (K) := K_1 -K_2 + K_3 -K  , \   \Omega_3 (K) := |K_1 |^2 - |K_2 |^2 + |K_3 |^2 - |K|^2 ,
\end{equation*}
and $\mathcal{C} (g(t, \cdot))$ is a correction term defined in \eqref{Th5-BGHS1-eq1} and satisfies the bound \eqref{Th5-BGHS1-eq2} in below, $\zeta (n)$ is the Riemann zeta function. Moreover, we shall only consider solutions of \eqref{CR} and \eqref{CR-n=2} with the following constraint
\begin{equation}
g (K) = K G (K) , \text{ for some scalar function } G (K) .
\end{equation}
This constraint ensures that $\check{g}$ is curl free, i.e. $\frac{K(j)}{|K|^2} K \cdot g(K) =g_j (K)$.



We define the functional spaces given by the following norms:
\begin{equation}
\begin{split}
& \|f \|_{X^l (\mathbb{R}^d)} :=  \| \la K \ra^l f(K) \|_{L^\infty (\mathbb{R}^d)} ,  \\
& \|f \|_{X^{l, N} (\mathbb{R}^d)} := \sum_{0 \leq |\alpha| \leq N} \| \nabla^\alpha  f \|_{X^l (\mathbb{R}^d)}   ,  \\
& X^{l}_n (\mathbb{R}^d) := (X^{l} (\mathbb{R}^d))^n , \\
& X^{l, N}_n (\mathbb{R}^d) := (X^{l, N} (\mathbb{R}^d))^n  \\
\end{split}
\end{equation}
and similarly define the spaces $X^l (\mathbb{Z}^d_L)$, $X^{l, N} (\mathbb{Z}^d_L)$, $X^{l}_n (\mathbb{Z}^d_L)$ and $X^{l, N}_n (\mathbb{Z}^d_L)$. 

Fix $0 < \gamma <1$. Let 
\[   Z_n (L) := \begin{cases} 
           \frac{1}{\zeta(2)} L^2 \log L , \text{ when } n =2 , \\
          \frac{\zeta(n-1)}{\zeta(n)} L^{2n-2}  , \text{ when } n \geq 3 ,
       \end{cases}
\]
\[   \delta (L) := \begin{cases} 
          ( \log L )^{-1} , \ \text{ when } n =2 , \\
           L^{-1+\gamma}  ,  \ \ \ \ \ \ \text{ when } n \geq 3 .
       \end{cases}
\]
Here $\zeta (n)$ is the Riemann zeta function.

We describe our main results as follows:

\begin{theorem} \label{th1}
Fix $l>2n$ and $0 < \gamma <1$. Let $g_0 \in X^{l+n+2, 3n+3}_n (\mathbb{R}^n)$, and suppose that $g(t, \xi)$ is a solution of \eqref{CR} over a time interval $[0, M]$ with initial data $g_0 = g(t=0)$. Denote
\begin{equation*}
B := \sup_{t \in [0, M] } \|g(t)\|_{ X^{l+n+2, 3n+3}_n (\mathbb{R}^n) } .
\end{equation*}
Let $u$ be a solution of \eqref{VNLS-rescaled} with initial data $u_0 = \frac{1}{L^n} \sum_{\mathbb{Z}_L^n} g_0 (K) e(K \cdot x)$, and set for $K \in \mathbb{Z}_L^n$
\begin{equation*}
a_K (t) := u_K (t) e(-|K|^2 t) .
\end{equation*} 
Then for $L$ sufficiently large and $\epsilon^2 L^\gamma$ sufficiently small, there exists a constant $C_{\gamma, M , B}$, such that for all $t \in [0, MT_R]$,
\begin{equation}
\Big\| a_K(t) - g(\frac{t}{T_R}, K)  \Big\|_{X^l_n (\mathbb{Z}_L^n)} \lesssim_n C_{\gamma, M, B} (\delta (L) + \epsilon^2 L^\gamma) ,
\end{equation}
where
\begin{equation}
T_R := \frac{L^{2n}}{\epsilon^2 Z_n (L)} .
\end{equation}
\end{theorem}

When $n=2$, Theorem \ref{th1} can be improved so that the term $(\log L)^{-1}$ in the error can be replaced by a polynomially decaying term:

\begin{theorem} \label{th2}
Let $n=2$. Fix $l>4$ and $0 < \gamma <1$. Let $g_0 \in X^{l+6, 15}_2 (\mathbb{R}^2)$, and suppose that $g(t, \xi)$ is a solution of \eqref{CR-n=2} over a time interval $[0, M]$ with initial data $g_0 = g(t=0)$. Denote
\begin{equation*}
B := \sup_{t \in [0, M] } \|g(t)\|_{ X^{l+6, 15}_2 (\mathbb{R}^2) } .
\end{equation*}
Let $u$ be a solution of \eqref{VNLS-rescaled} with initial data $u_0 = \frac{1}{L^2} \sum_{\mathbb{Z}_L^2} g_0 (K) e(K \cdot x)$, and set for $K \in \mathbb{Z}_L^2$
\begin{equation*}
a_K (t) := u_K (t) e(-|K|^2 t) .
\end{equation*} 
Then for $L$ sufficiently large and $\epsilon^2 L^\gamma$ sufficiently small, there exists a constant $C_{\gamma, M , B}$, such that for all $t \in [0, MT_R]$,
\begin{equation}
\Big\| a_K(t) - g(\frac{t}{T_R}, K)  \Big\|_{X^l_2 (\mathbb{Z}_L^2)} \lesssim C_{\gamma, M, B} ( \frac{1}{L^{1/3}-\gamma} + \epsilon^2 L^\gamma) ,
\end{equation}
where
\begin{equation}
T_R := \frac{L^{4}}{\epsilon^2 Z_2 (L)} = \frac{\zeta (2) L^2}{\epsilon^2 \log L} .
\end{equation}
\end{theorem}

The key tool of the proofs of Theorem \ref{th1} and Theorem \ref{th2} are the estimates on lattice sums and resonant sums obtained by utilizing the Hardy-Littlewood circle method, as shown in \cite{BGHS1}, see Section \ref{S:Prelim-est}. 

We also consider the following coupled Nonlinear Schr\"odinger Equations (NLS) with an $n$-dimensional spatial variable and $d$ components:
\begin{equation} \label{NLS}
-i\partial_t E + \frac{1}{2 \pi} \Delta E =  |E|^2 E , \ E|_{t=0} = \epsilon E_0 ,
\end{equation}
with $E = E(x, t) : \mathbb{T}^n_L \times \mathbb{R} \rightarrow \mathbb{C}^d$ being a complex-valued vector field, and $E_0 = E_0(x) : \mathbb{T}^n_L  \rightarrow \mathbb{C}^d$ is a fixed complex-valued vector field. By the ansatz $E= \epsilon u$ which represents the selection of a small data of size $\sim \epsilon$, the equation for $u$ becomes
\begin{equation} \label{NLS-rescaled}
-i\partial_t u + \frac{1}{2 \pi} \Delta u = \epsilon^2   (|u|^2 u) , \ u|_{t=0} =  E_0 ,
\end{equation}

We can obtain results that are parallel to Theorem \ref{th1} and Theorem \ref{th2} for \eqref{NLS}:

\begin{theorem} \label{th1-NLS}
Fix $l>2n$ and $0 < \gamma <1$. Let $g_0 \in X^{l+n+2, 3n+3}_d (\mathbb{R}^n)$, and suppose that $g(t, \xi)$ is a solution of the following \emph{continuous resonant equation}
\begin{equation} \label{CR-NLS}
-i\partial_t g(t,K)= \mathcal{T} (g(t, \cdot) , g(t, \cdot) , g(t, \cdot)) (t, K) , \ K \in \mathbb{R}^n 
\end{equation}
over a time interval $[0, M]$ with initial data $g_0 = g(t=0)$. Here $\mathcal{T} = (\mathcal{T}_j )_{j=1, 2, \cdots , d} $, 
\begin{equation*}
\mathcal{T}_j (g, g, g) (K) :=  \int_{\mathbb{R}^{3n}} \big(\sum_{l=1}^d g_l (K_1) \overline{g}_l (K_2) \big) g_j (K_3) \delta_{\mathbb{R}^n} (S_3 (K)) \delta_{\mathbb{R}} (\Omega_3 (K)) dK_1 dK_2 dK_3  .
\end{equation*}
Denote
\begin{equation*}
B := \sup_{t \in [0, M] } \|g(t)\|_{ X^{l+n+2, 3n+3}_n (\mathbb{R}^n) } .
\end{equation*}
Let $u$ be a solution of \eqref{NLS-rescaled} with initial data $u_0 = \frac{1}{L^n} \sum_{\mathbb{Z}_L^n} g_0 (K) e(K \cdot x)$, and set for $K \in \mathbb{Z}_L^n$
\begin{equation*}
a_K (t) := u_K (t) e(-|K|^2 t) .
\end{equation*} 
Then for $L$ sufficiently large and $\epsilon^2 L^\gamma$ sufficiently small, there exists a constant $C_{\gamma, M , B}$, such that for all $t \in [0, MT_R]$,
\begin{equation}
\Big\| a_K(t) - g(\frac{t}{T_R}, K)  \Big\|_{X^l_n (\mathbb{Z}_L^n)} \lesssim_{n, d} C_{\gamma, M, B} (\delta (L) + \epsilon^2 L^\gamma) ,
\end{equation}
where
\begin{equation}
T_R := \frac{L^{2n}}{\epsilon^2 Z_n (L)} .
\end{equation}
\end{theorem}

Similarly as Theorem \ref{th2}, when $n=2$, Theorem \ref{th1-NLS} can be improved to obtain:

\begin{theorem} \label{th2-NLS}
Let $n=2$. Fix $l>4$ and $0 < \gamma <1$. Let $g_0 \in X^{l+6, 15}_d (\mathbb{R}^2)$, and suppose that $g(t, \xi)$ is a solution of the following \emph{continuous resonant equation}
\begin{equation} \label{CR-NLS-n=2}
-i\partial_t g(t,K)= \mathcal{T} (g(t, \cdot) , g(t, \cdot) , g(t, \cdot)) (t, K) + \frac{\zeta(2)}{\log L} \mathcal{C} (g(t, \cdot)) (t, K) , \ K \in \mathbb{R}^2 .
\end{equation}
over a time interval $[0, M]$ with initial data $g_0 = g(t=0)$. Here $\mathcal{T} = (\mathcal{T}_j )_{j=1, 2, \cdots , d} $, 
\begin{equation*}
\mathcal{T}_j (g, g, g) (K) :=   \int_{\mathbb{R}^{3n}} \big(\sum_{l=1}^d g_l (K_1) \overline{g}_l (K_2) \big) g_j (K_3) \delta_{\mathbb{R}^n} (S_3 (K)) \delta_{\mathbb{R}} (\Omega_3 (K)) dK_1 dK_2 dK_3  , 
\end{equation*} 
and $\mathcal{C} (g(t, \cdot))$ is a correction term defined in \eqref{Th5-BGHS1-eq1} and satisfies \eqref{Th5-BGHS1-eq2} in below. Denote
\begin{equation*}
B := \sup_{t \in [0, M] } \|g(t)\|_{ X^{l+6, 15}_d (\mathbb{R}^2) } .
\end{equation*}
Let $u$ be a solution of \eqref{NLS-rescaled} with initial data $u_0 = \frac{1}{L^2} \sum_{\mathbb{Z}_L^2} g_0 (K) e(K \cdot x)$, and set for $K \in \mathbb{Z}_L^2$
\begin{equation*}
a_K (t) := u_K (t) e(-|K|^2 t) .
\end{equation*} 
Then for $L$ sufficiently large and $\epsilon^2 L^\gamma$ sufficiently small, there exists a constant $C_{\gamma, M , B}$, such that for all $t \in [0, MT_R]$,
\begin{equation}
\Big\| a_K(t) - g(\frac{t}{T_R}, K)  \Big\|_{X^l_d (\mathbb{Z}_L^2)} \lesssim_{d} C_{\gamma, M, B} ( \frac{1}{L^{1/3}-\gamma} + \epsilon^2 L^\gamma) ,
\end{equation}
where
\begin{equation}
T_R := \frac{L^{4}}{\epsilon^2 Z_2 (L)} = \frac{\zeta (2) L^2}{\epsilon^2 \log L} .
\end{equation}
\end{theorem}


Another goal of this paper is to investigate various properties of the continuous resonant equations, including the Hamiltonian structure and the well-posedness, etc.. We will point out that the CR equations have Hamiltonian structures parallel to the structure of \eqref{VNLS} and \eqref{NLS}. We are also going to observe symmetries, the corresponding conservation laws as well as some well-posedness properties for the CR equations. Moreover, we will study the stationary solutions to the CR equations in a variational point of view, yielding some existence results. These directions have been pursued in varied literature for the CR equation for the NLS equation with one component $-i\partial_t u + \frac{1}{2 \pi} \Delta u =  |u|^2 u $, $u= u(t, x) \in \mathbb{C}$, $x \in \mathbb{R}^n$, see, for example, \cite{BGHS2}, \cite{GHTho1}, \cite{GHTho2}, \cite{Fennell1}.   



The contents in the paper are arranged as follows. In Section \ref{S:Prelim-est}, we give some lemmas which are key tools on evaluating the sums on the lattices and the resonant terms, which are provided in \cite{BGHS1}. Then, we perform normal form transformation to eliminate all the non-resonant term in Section \ref{S:Normal-form-transformation}. The proof of the main results Theorem \ref{th1} and Theorem \ref{th2} are given in Section \ref{S:Proof-main-th}, utilizing the results in Section \ref{S:Prelim-est} and Section \ref{S:Normal-form-transformation}. In Section \ref{S:General-case-p}, we consider a generalized version of \eqref{VNLS} whose nonlinearity is of power $2p+1$ (see \eqref{VNLS-p}), and derive the CR equation for it, see Theorem \ref{th3} and Theorem \ref{th4}, which are parallel to Theorem \ref{th1} and Theorem \ref{th2}, respectively. Section \ref{S:CR-analysis} and Section \ref{S:CR-NLS-analysis} are devoted to analysis on the CR equations \eqref{CR} and \eqref{CR-NLS}, respectively. These include the Hamiltonian structures, symmetries and conserved quantities, local well-posedness and results on stationary waves. Moreover, for \eqref{CR-NLS}, we briefly discuss the dynamics of the equation on the eigenspaces of the harmonic oscillator when $n=2$. 



\section{Preliminary Estimates}
\label{S:Prelim-est}


We recall the following estimates on lattice sums and resonant sums obtained in \cite{BGHS1}:

\begin{proposition} [Theorem 4 in \cite{BGHS1}]
\label{Th4-BGHS1}
Let $K \in \mathbb{Z}_L^n$, and 
\begin{equation*}
R_\mu (K) := \{ K_i \in \mathbb{Z}_L^n ; \ S_3 (K) = K_1 -K_2 + K_3 -K =0 , \   \Omega_3 (K) = |K_1 |^2 - |K_2 |^2 + |K_3 |^2 -|K|^2 =\mu \} .
\end{equation*} 
Given sequences $\{ a_K \}$, $\{ b_K \}$ and $\{ c_K \}$, such that
\begin{equation*}
|a_K|+ |b_K| + |c_K| \lesssim \la K \ra^{-l} ,
\end{equation*}
we have, for $l > 3n+2$,
\[ \sup_{K, \mu} \la K \ra^{l} \sum_{R_\mu (K)} a_{K_1} b_{K_2} c_{K_3} \lesssim Z_n (L) \lesssim  \begin{cases} 
          L^2 \log L  \  \text{ if } n=2, \\
          L^{2n-2}  \ \ \ \ \text{ if } n \geq 3 .
       \end{cases}
\]
\end{proposition}

\begin{proposition} [Theorem 5 in \cite{BGHS1}]
\label{Th5-BGHS1}
Let  $f_1$, $f_2$, $f_3 \in X^{l+n+2, N} (\mathbb{R}^n)$ for $l > 2n$ and $N > 3n+2$, and set 
\begin{equation}
W(z) = f_1 (K+z_1) \overline{f}_2 (K+z_1+z_2) f_3 (K+z_2) , \  z = (z_1, z_2) \in \mathbb{Z}_L^{2n}
\end{equation}
and
\begin{equation}
\mathcal{T}(W)(K) = \mathcal{T}(f_1, f_2, f_3) (K) := \int_{\mathcal{R}^{2n}} \delta(\omega (z)) W(z) dz ,
\end{equation}
with $\omega(z) = z_1 \cdot z_2 $. Then \\
1) For $n =2$, define
\begin{equation} \label{Th5-BGHS1-eq1}
\tilde{\Delta} (W) = \frac{1}{Z_2 (L)} \sum_{K_1 \cdot K_3 =0} W(K_1, K_3) - (\mathcal{T} (W) + \frac{\zeta(2)}{\log L}  \mathcal{C} (W)) ,
\end{equation}
where $\mathcal{C} (W)$ is a correction operator independent of $L$. 
If $f_j \in X^{l+6, N} (\mathbb{R}^2)$ for $l >4 $ and $N >14$, then
\begin{equation} \label{Th5-BGHS1-eq2}
\begin{split}
& \|\mathcal{C} (W)) \|_{X^l} \lesssim \prod_{j=1}^3 \|f_j\|_{X^{l+6,N} (\mathbb{R}^2)} , \\
& \|\tilde{\Delta}(W)) \|_{X^l} \lesssim L^{-1/3+} \prod_{j=1}^3 \|f_j\|_{X^{l+6,N} (\mathbb{R}^2)} . \\
\end{split}
\end{equation}
2) For $n \geq 3$, define
\begin{equation} \label{Th5-BGHS1-eq3}
\Delta (W) = \frac{1}{Z_n (L)} \sum_{z\in \mathbb{Z}_L^{2n}, \omega (z) =0} W(z) - \mathcal{T} (W) ,
\end{equation}
then
\[ \|\Delta (W) \|_{X^l} \lesssim  \prod_{j=1}^3 \|f_j\|_{X^{l+n+2,N} (\mathbb{R}^n)} 
  \begin{cases} 
          L^{-1} \log L  \  \text{ if } n=3, \\
          L^{-1}  \ \ \ \ \text{ if } n \geq 4 .
       \end{cases}
\]
\end{proposition}


\begin{proposition} [Theorem 6 in \cite{BGHS1}]
\label{Th6-BGHS1}
Let $p \in \mathbb{N}$, $f_j \in X^{l, N} (\mathbb{R}^n)$, and denote $z = (\mathbf{J}_e , \mathbf{J}_o )$, where $\mathbf{J}_e = (J_1, J_3, \cdots , J_{2p-1}) \in \mathbb{Z}^{np}_L$, $\mathbf{J}_o = (J_2, J_4, \cdots , J_{2p}) \in \mathbb{Z}^{np}_L$, and $J_i$'s are given by
\begin{equation} \label{Th6-BGHS1-eq0}
\begin{split}
& J_1 = K_1 + K_3  \cdots - K_{2p} , \ J_2 = K_2 - K_1 , \\
& J_{2i-1} = K_{2i-1} + K_{2i+1} -K_{2i+2} \cdots - K_{2p} , \quad J_{2i} = K_{2i} - K_{2i-1} , \\
& J_{2p-1} = K_{2p-1} . \\
\end{split}
\end{equation}
Denote $W(z) = ( W_j (z) )_{j=1, 2, \cdots , n} $, and
\begin{equation}
\begin{split}
W_j (z) 
& =  \frac{K(j)}{|K|^2} \sum_{m=1}^n K(m) \big(\sum_{l=1}^n f_{1, l} (K + K_1)  \overline{f}_{2, l} (K + K_2) \big) \\
& \quad \cdots \big(\sum_{l=1}^n f_{2p-1, l} (K + K_{2p-1})  \overline{f}_{2p, l} (K + K_{2p} ) \big) \cdot f_{2p+1, m} (K + K_{2p+1}) , \\
\end{split}
\end{equation}
where $K_i$ are considered as functions of $J_i$ by inverting \eqref{Th6-BGHS1-eq0}, and 
\begin{equation}
\mathcal{P}(W)(K) = \int_{\mathbb{R}^{n(2p +1)}} \delta (\Omega_{2p+1} (K))   \delta (S_{2p+1} (K)) W(z) dK_1 dK_2 \cdots dK_{2p+1} .
\end{equation}
Set 
\[   Z_{pn} (L) := \begin{cases} 
           \frac{1}{\zeta(2)} L^2 \log L , \text{ when } pn =2 , \\
          \frac{\zeta(pn-1)}{\zeta(pn)} L^{2pn-2}  , \text{ when } pn \neq 2 .
       \end{cases}
\]
1) For $pn =2$, define
\begin{equation} \label{Th6-BGHS1-eq1}
\tilde{\Delta} (W) = \frac{1}{Z_2 (L)} \sum_{ \mathbf{J}_o \cdot \mathbf{J}_e =0 } W(z) - (\mathcal{P} (W) + \frac{\zeta(2)}{\log L}  \mathcal{C} (W)) ,
\end{equation}
where $\mathcal{C} (W)$ is a correction operator independent of $L$. 
If $f_j \in X^{l+6, 11} (\mathbb{R}^2)$, then
\begin{equation} \label{Th6-BGHS1-eq2}
\begin{split}
& \|\mathcal{C} (W)) \|_{X^l} \lesssim \prod_{j=1}^{2p+1} \|f_j\|_{X^{l+6,11} (\mathbb{R}^2)} , \\
& \|\tilde{\Delta}(W)) \|_{X^l} \lesssim L^{-1/3+} \prod_{j=1}^{2p+1} \|f_j\|_{X^{l+6,11} (\mathbb{R}^2)} . \\
\end{split}
\end{equation}
2) For $pn \neq 2$, define
\begin{equation} \label{Th6-BGHS1-eq3}
\Delta (W) = \frac{1}{Z_n (L)} \sum_{z\in \mathbb{Z}_L^{2n}, \ \omega (z) =0} W(z) - \mathcal{P} (W) .
\end{equation}
If $f_j \in X^{l+n+2, 4n+2} (\mathbb{R}^n)$, then
\begin{equation} \label{Th6-BGHS1-eq4}
\begin{split}
& \| \Delta(W)) \|_{X^l} \lesssim L^{-1+} \prod_{j=1}^{2p+1} \|f_j\|_{X^{l+n+2,4n+2} (\mathbb{R}^n)} . \\
\end{split}
\end{equation}
\end{proposition}

\section{Normal Form Transformation}
\label{S:Normal-form-transformation}

In order to prove Theorem \ref{th1} and \ref{th2}, we first perform a normal form transformation on \eqref{VNLS}. We denote
\begin{equation*}
\begin{split}
& S_{2d+1} (K) = K_1 - K_2 + K_3 - \cdots +K_{2d+1} -K , \\
&  \Omega_{2d+1} (K) = |K_1|^2 - |K_2|^2 + |K_3|^2 - \cdots +|K_{2d+1}|^2 - |K|^2 . \\
\end{split} 
\end{equation*}
Let
\begin{equation} \label{D:v}
v_j = u_j+  \sum_{d=1}^P \epsilon^{2d} H_{2d+1}^j (u) , \ j=1, 2, \cdots , n , 
\end{equation}
where $H_{2d+1}^j$ ($j = 1, 2, \cdots , n$) is a $2d+1$ multi-linear form in $u_l$ (in odd entries) and $\overline{u}_l$ (in even entries), $l= 1, 2, \cdots, n$. 

Recall that $u$ solves \eqref{VNLS}. We have, $v_j$ satisfies (notice that $ \partial_t  \sum_{d=1}^P \epsilon^{2d} H_{2d+1}^j (u) =  \sum_{k=1}^n \sum_{d=1}^P \epsilon^{2d} \frac{ \delta H_{2d+1}^j}{\delta u} (\partial_t u_k)$)
\begin{equation*}
\begin{split}
& -i\partial_t v_j + \frac{1}{2 \pi} \Delta v_j = \epsilon^2 \big[ \Delta^{-1} \nabla \nabla \cdot (|u|^2 u)\big]_j + \sum_{d=1}^P \frac{\epsilon^{2d}}{2 \pi} \Delta H_{2d+1}^j (u) \\
& \qquad \qquad \qquad + \sum_{k=1}^n \sum_{d=1}^P \epsilon^{2d} \frac{\delta H_{2d+1}^j}{\delta u_k} (\frac{-1}{2\pi} \Delta u_k) + \sum_{k=1}^n  \sum_{d=1}^P \epsilon^{2d+2} \frac{\delta H_{2d+1}^j}{\delta u_k} \big[ \Delta^{-1} \nabla \nabla \cdot (|u|^2 u)\big]_k , \\
\end{split}
\end{equation*}
where we used the notation for any $F$ depending on $u_j$ and $\overline{u}_j$:
\begin{equation*}
\frac{\delta F}{\delta u_j} (w) := \frac{\partial F}{\partial u_j} w  + \frac{\partial F}{\partial \overline{u}_j} \overline{w} . 
\end{equation*}
Writing $\mathcal{L} H_{2d+1}^j (u) = \frac{1}{2\pi} \Delta H_{2d+1}^j (u) - \sum_{k=1}^n \frac{\delta H_{2d+1}^j}{ \delta u_k} (\frac{1}{2\pi}\Delta u_k)$, and collecting terms of the same order in $\epsilon$, we conclude that   
\begin{equation*}
\begin{split}
& \quad -i\partial_t v_j + \frac{1}{2 \pi} \Delta v_j \\
& = \epsilon^2 \big[ \Delta^{-1} \nabla \nabla \cdot (|u|^2 u)\big]_j  + \epsilon^2 \mathcal{L} H_3^j (u)  + \sum_{d=2}^P \epsilon^{2d} \big( \mathcal{L} H_{2d+1}^j (u)  + \sum_{k=1}^n \frac{\delta H_{2d-1}^j}{\delta u_k} \big[ \Delta^{-1} \nabla \nabla \cdot (|u|^2 u)\big]_k  \big) \\
& \quad + \epsilon^{2(P+1)} \sum_{k=1}^n \frac{\delta H_{2P+1}^j}{\delta u_k} \big[ \Delta^{-1} \nabla \nabla \cdot (|u|^2 u)\big]_k . \\
\end{split}
\end{equation*}
We expand $u$ and $v$ in their Fourier coefficients
\begin{equation*}
u_j(t, x) = \frac{1}{L^n} \sum_{K \in \mathbb{Z}_L^n} u_{j,K} (t) e(K \cdot x) ,  \ v_j(t, x) = \frac{1}{L^n} \sum_{K \in \mathbb{Z}_L^n} v_{j,K} (t) e(K \cdot x) ,
\end{equation*}
where $K \in \mathbb{Z}_L^n = (\frac{\mathbb{Z}}{L})^n$ and $e(\alpha) = \exp (2 \pi i \alpha)$. Denote 
\begin{equation*}
a_{j,K} (t) :=  u_{j,K} (t) e(-|K|^2 t) ,  \  b_{j,K} (t) :=  v_{j,K} (t) e(-|K|^2 t) .
\end{equation*}
Moreover, we write the multi-linear forms $H^j_{2d+1}$ ($j = 1, 2, \cdots , n$) as 
\begin{equation*}
H_{2d+1}^j (u) = \frac{1}{L^n} \sum_{K \in \mathbb{Z}_L^n} H_{2d+1, K}^j (u)  e(K \cdot x) , 
\end{equation*}
where
\begin{equation*}
\begin{split}
H_{2d+1, K}^j (u) = 
& \frac{1}{L^{2nd}} \sum_{S_{2d+1} (K) =0} \sum_{\substack{J= (j_1, j_2, \cdots , j_{2d+1}) \\ \in \{1, 2, \cdots , n \}^{2d+1}}}  h_{J, 2d+1}^j (K, K_1, K_2, \cdots , K_{2d+1}) \\
& \cdot u_{j_1, K_1} \overline{u}_{j_2, K_2} \cdots u_{j_{2d+1},K_{2d+1}} . \\
\end{split} 
\end{equation*}
Here $J= (j_1, j_2, \cdots , j_{2d+1}) \in \{1, 2, \cdots , n \}^{2d+1}$, $u_{l, K_j} = u_{l, K_j} (t  )$.

The equation for $v_{j,K}$ can be written as 
\begin{equation} \label{vjK-eqn}
\begin{split}
& -i\partial_t v_{j, K} - 2 \pi |K|^2 v_{j, K} = \frac{\epsilon^2}{L^{2n}}  \frac{K(j)}{|K|^2} \sum_{m=1}^n K(m) \sum_{S_3 (K)=0 } (\sum_{l=1}^n u_{l, K_1} \overline{u}_{l, K_2} )  u_{m, K_3}+ \epsilon^2 \mathcal{L} H_{3, K}^j(u) \\
& \qquad \qquad + \sum_{d=2}^P \epsilon^{2d} \big(\mathcal{L} H_{2d+1, K}^j(u) + \sum_{k=1}^n \frac{\delta H_{2d-1, K}^j}{\delta u_k} \big[ \Delta^{-1} \nabla \nabla \cdot (|u|^2 u)\big]_k  \big) \\
& \qquad \qquad + \epsilon^{2(P+1)} \sum_{k=1}^n \frac{\delta H_{2P+1, K}^j}{\delta u_k} \big[ \Delta^{-1} \nabla \nabla \cdot (|u|^2 u)\big]_k , \\
\end{split}
\end{equation}
where $u_{l, K_j} = u_{l, K_j} (t  )$, and $K (l)$ is the $l$-th component in $K$, $K_j (l)$ is the $l$-th component in $K_j$ . Moreover, $\mathcal{L} H_{2d+1, K}^j(u)$ and $\frac{\delta H_{2d-1, K}^j}{\delta u_k} \big[ \Delta^{-1} \nabla \nabla \cdot (|u|^2 u)\big]_k$ are given by
\begin{equation*}
\begin{split}
\mathcal{L} H_{2d+1, K}^j(u) =
& \frac{1}{L^{2nd}} \sum_{S_{2d+1} (K) =0} \sum_{\substack{J= (j_1, j_2, \cdots , j_{2d+1}) \\ \in \{1, 2, \cdots , n \}^{2d+1}}}  2 \pi \Omega_{2d+1} (K, K_1, \cdots , K_{2d+1}) \\
& \cdot h_{J, 2d+1}^j (K, K_1, K_2, \cdots , K_{2d+1})  u_{j_1, K_1} \overline{u}_{j_2, K_2} \cdots u_{j_{2d+1},K_{2d+1}} , \\
\end{split}
\end{equation*}
\begin{equation*}
\begin{split}
& \quad \sum_{k=1}^n \frac{\delta H_{2d-1, K}^j}{\delta u_k} \big[ \Delta^{-1} \nabla \nabla \cdot (|u|^2 u)\big]_k (K) \\
& =  \frac{1}{L^{2nd}} \sum_{S_{2d-1} (K) =0}  \sum_{\substack{J= (j_1, j_2, \cdots , j_{2d-1}) \\ \in \{1, 2, \cdots , n \}^{2d-1}}}  h_{J, 2d-1}^j (K, K_1, \cdots , K_{2d-1}) \\
& \quad \cdot \big( \frac{K_1(j_1)}{|K_1|^2} \sum_{m=1}^n K_1(m) \sum_{N-K_{2d} +K_{2d+1} =K_1 } (\sum_{l=1}^n u_{l, N} \overline{u}_{l, K_{2d}} )  u_{m, K_{2d+1}} \big)  \\
& \quad \cdot   \overline{u}_{j_2, K_2} u_{j_3, K_3} \cdots u_{j_{2d-1},K_{2d-1}}  \\
& \quad + \frac{1}{L^{2nd}} \sum_{S_{2d-1} (K) =0}  \sum_{\substack{J= (j_1, j_2, \cdots , j_{2d-1}) \\ \in \{1, 2, \cdots , n \}^{2d-1}}}  h_{J, 2d-1}^j (K, K_1, \cdots , K_{2d-1}) \\
& \quad \cdot \big( \frac{K_2(j_2)}{|K_2|^2} \sum_{m=1}^n K_2(m) \sum_{N-K_{2d} +K_{2d+1} =K_2 } (\sum_{l=1}^n \overline{u}_{l, N} u_{l, K_{2d}} )  \overline{u}_{m, K_{2d+1}} \big)  \\
& \quad \cdot  u_{j_1, K_1} u_{j_3, K_3} \cdots u_{j_{2d-1},K_{2d-1}}  \\
& \quad + \cdots \\
\end{split}
\end{equation*}
In particular,
\begin{equation*}
\begin{split}
\mathcal{L} H_{3, K}^j(u) =
& \frac{1}{L^{2n}} \sum_{S_{3} (K) =0} \sum_{\substack{J= (j_1, j_2, j_3) \\ \in \{1, 2, \cdots , n \}^3}}  2 \pi \Omega_{3} (K, K_1, K_2 , K_3) \\
& \cdot h_{J, 3}^j (K, K_1, K_2, K_3)  u_{j_1, K_1} \overline{u}_{j_2, K_2} u_{j_{3},K_{3}} . \\
\end{split}
\end{equation*}
We iteratively define $H_{2d+1}^j$ to eliminate the non-resonant terms (the normal form transformation) as follows:
\begin{equation}
\begin{split}
H_{3, K}^j (u) =
& \frac{-1}{L^{2n}} \sum_{\substack{S_3(K)=0, \\  \Omega_3(K) \neq 0}} \frac{1}{ 2 \pi \Omega_{3} (K, K_1, K_2 , K_3)}  \frac{K(j)}{|K|^2} \sum_{m=1}^n K(m)   (\sum_{l=1}^n u_{l, K_1} \overline{u}_{l, K_2} ) u_{m, K_3} , \\
\end{split}
\end{equation}
\begin{equation}  \label{H_2d+1-eqn}
\begin{split}
H_{2d+1, K}^j (u) 
& =  \frac{-1}{L^{2nd}} \sum_{S_{2d-1} (K) =0} \sum_{\substack{J= (j_1, j_2, \cdots , j_{2d-1}) \\ \in \{1, 2, \cdots , n \}^{2d-1}}}  h_{J, 2d-1}^j (K, K_1, \cdots , K_{2d-1}) \\
& \quad \cdot \big( \sum_{\substack{N-K_{2d} +K_{2d-1} =K_1, \\ \Omega_{2d+1}(K, N, K_2, \cdots , K_{2d+1}) \neq 0}} \frac{1}{2\pi \Omega_{2d+1}(K, N, K_2, \cdots , K_{2d+1})} \\
& \quad \cdot   \frac{K_1(j_1)}{|K_1|^2} \sum_{m=1}^n K_1(m)  (\sum_{l=1}^n u_{l, N} \overline{u}_{l, K_{2d}} )  u_{m, K_{2d+1}} \big)  \\
& \quad \cdot \overline{u}_{j_2, K_2} u_{j_3, K_3} \overline{u}_{j_4, K_4} u_{j,K_5} \overline{u}_{j,K_6} \cdots u_{j,K_{2d-1}}  \\
& \quad + \frac{-1}{L^{2nd}} \sum_{S_{2d-1} (K) =0}  \sum_{\substack{J= (j_1, j_2, \cdots , j_{2d-1}) \\ \in \{1, 2, \cdots , n \}^{2d-1}}}  h_{J, 2d-1}^j (K, K_1, \cdots , K_{2d-1}) \\
& \quad \cdot \big( \sum_{\substack{N-K_{2d} +K_{2d-1} =K_2, \\ \Omega_{2d+1}(K, K_1,   N, \cdots , K_{2d+1}) \neq 0}} \frac{1}{2\pi \Omega_{2d+1}(K, K_1,  N, \cdots , K_{2d+1})} \\
& \quad \cdot \big( \frac{K_2(j_2)}{|K_2|^2} \sum_{m=1}^n K_2(m)  (\sum_{l=1}^n \overline{u}_{l, N} u_{l, K_{2d}} )  \overline{u}_{m, K_{2d+1}} \big)  \\
& \quad \cdot  u_{j_1, K_1} u_{j_3, K_3} \cdots u_{j_{2d-1},K_{2d-1}}  \\
& \quad + \cdots  \\
\end{split}
\end{equation}
With this choice of $H_{2d+1}^j$, the equation \eqref{vjK-eqn} becomes 
\begin{equation} \label{vjK-eqn-res}
\begin{split}
& -i\partial_t v_{j, K} - 2 \pi |K|^2 v_{j, K} = \frac{\epsilon^2}{L^{2n}}   \frac{K(j)}{|K|^2} \sum_{m=1}^n K(m) \sum_{S_3 (K)=0 , \Omega_3 (K) = 0 } (\sum_{l=1}^n u_{l, K_1} \overline{u}_{l, K_2} )  u_{m, K_3} \\
& \qquad \qquad + \sum_{k=1}^n \sum_{d=2}^P \epsilon^{2d}  \sum_{S_{2d+1} (K)=0, \Omega_{2d+1} (K) = 0 }  \frac{\delta H_{2d-1, K}^j}{\delta u_k} \big[ \Delta^{-1} \nabla \nabla \cdot (|u|^2 u)\big]_k  \\
& \qquad \qquad + \epsilon^{2(P+1)}  \sum_{k=1}^n \frac{\delta H_{2P+1, K}^j}{\delta u_k} \big[ \Delta^{-1} \nabla \nabla \cdot (|u|^2 u)\big]_k . \\
\end{split}
\end{equation}
Here $\sum_{S_{2d+1} (K)=0, \Omega_{2d+1} (K) = 0 }$ means that we only take the parts that satisfies the resonance condition.

Now, using $\frac{|K(j)K(m)|}{|K|^2} \leq 1$ together with Proposition \ref{Th4-BGHS1}, we have the following (here all the constants involved in $\lesssim$, $\lesssim_n$, $\cdots$ depend on $P$)

\begin{lemma} \label{L:lm13-BGHS1}
For any $d \leq P$, 
\begin{equation}
\|H_{2d+1}^j (u)\|_{X^l} \lesssim_n L^+ \|u\|^{2d+1}_{X^l_n} .
\end{equation}
As a result, for $v$ defined in \eqref{D:v}, there holds
\begin{equation}
\|v_j-u_j\|_{X^l} \lesssim_n L^+ \sum_{d=1}^P \epsilon^{2d} \|u\|^{2d+1}_{X^l_n} , \ j = 1, 2, \cdots , n .
\end{equation}
(Here $L^+$ means the power $L^{+ \epsilon}$ for some $\epsilon >0$.)
\end{lemma}

\begin{proof}

It suffices to show that for each $J= (j_1, j_2, \cdots , j_{2d+1}) \in \{1, 2, \cdots , n \}^{2d+1}$,
\begin{equation}
\frac{1}{L^{2nd}} \sum_{S_{2d+1} (K)=0} |h_{J, 2d+1}^j (K, K_1, K_2, \cdots , K_{2d+1} )| \la K_1\ra ^{-l} \la K_2\ra ^{-l} \cdots \la K_{2d+1}\ra ^{-l} \lesssim_n L^+ \la K \ra ^{-l} .
\end{equation}
We start from the case $d=1$, where $|h_{J, 3}^j | \lesssim _n \frac{1}{  |\Omega_{3} (K, K_1, K_2 , K_3)|}  | \frac{K(j)}{|K|^2} \sum_{m=1}^n K(m) | \lesssim_n \frac{1}{|\Omega_3|}$. The proof of the case $d=1$ can then be done in essentially the same way as in Lemma 13 in \cite{BGHS1} by using $\frac{|K(j)K(m)|}{|K|^2} \leq 1$.

Next we turn to the case $d \geq 2$ and prove the bound inductively. We use the recursive definition \eqref{H_2d+1-eqn} to write
\begin{equation}
\begin{split}
& \quad  \frac{1}{L^{2nd}} \sum_{S_{2d+1} (K)=0} |h_{J, 2d+1}^j (K, K_1, K_2, \cdots , K_{2d+1} )| \la K_1\ra ^{-l} \la K_2\ra ^{-l} \cdots \la K_{2d+1}\ra ^{-l} \\
& \lesssim_n  \frac{1}{L^{2nd}} \sum_{S_{2d+1} (K)=0} |h_{J, 2d-1}^j | \Big| \sum_{\substack{N-K_{2d} +K_{2d-1} =K_1, \\ \Omega_{2d+1}(K_1, N, K_2, \cdots , K_{2d+1}) \neq 0}} \frac{1}{2\pi \Omega_{2d+1}(K_1, N, K_2, \cdots , K_{2d+1})} \\
& \quad \cdot   \frac{K_1(j_1)}{|K_1|^2} \sum_{m=1}^n K_1(m)  (\sum_{l=1}^n \la N \ra^{-l} \la K_{2d} \ra^{-l} )  \la K_{2d+1} \ra^{-l}   \Big|   \la K_2 \ra ^{-l} \cdots \la K_{2d-1} \ra ^{-l}  \\
& \lesssim_n  \frac{1}{L^{2nd}} \sum_{S_{2d+1} (K)=0} |h_{J, 2d-1}^j | \Big( \sum_{\substack{N-K_{2d} +K_{2d-1} =K_1, \\ \Omega_{2d+1}(K_1, N, K_2, \cdots , K_{2d+1}) \neq 0}} \frac{1}{|\Omega_{2d+1}|} \la N \ra ^{-l}    \la K_{2d} \ra ^{-l}    \la K_{2d+1} \ra ^{-l}   \Big)   \\
& \quad \cdot \la K_2 \ra ^{-l} \cdots \la K_{2d-1} \ra ^{-l} . \\
\end{split}
\end{equation}
Here we have used the fact that $\frac{|K(j)K(m)|}{|K|^2} \leq 1$. The inner sum can be bounded as in the $d=1$ case above, which leads to the estimate
\begin{equation}
\begin{split}
& \quad  \frac{1}{L^{2nd}} \sum_{S_{2d+1} (K)=0} |h_{J, 2d+1}^j (K, K_1, K_2, \cdots , K_{2d+1} )| \la K_1\ra ^{-l} \la K_2\ra ^{-l} \cdots \la K_{2d+1}\ra ^{-l} \\
& \lesssim_n  \frac{L^+}{L^{n(2d-2)}} \sum_{S_{2d-1} (K)=0} |h_{J, 2d-1}^j |  \la K_2 \ra ^{-l} \cdots \la K_{2d+1} \ra ^{-l}  \la K_1 \ra ^{-l} \cdots \la K_{2d-1} \ra ^{-l}  \\
& \lesssim_n   L^+   \la K \ra ^{-l}  , \\
\end{split}
\end{equation}
where in the last inequality we used the bound at the rank $d - 1$, which follows from the inductive assumption.
\end{proof}

\section{Proof of the Main Theorem}
\label{S:Proof-main-th}

We now give the proof of Theorem \ref{th1} and Theorem \ref{th2}. All the constants involved in $\lesssim$, $\lesssim_n$, $\cdots$ here depend on $P$.

\begin{proof} (Proof of Theorem \ref{th1} and Theorem \ref{th2})

Let us consider Theorem \ref{th1}. Recall 
\begin{equation*}
a_{j,K} (t) :=  u_{j,K} (t) e(-|K|^2 t) ,  \  b_{j,K} (t) :=  v_{j,K} (t) e(-|K|^2 t) .
\end{equation*}
From Lemma \ref{L:lm13-BGHS1}, we have
\begin{equation*}
\begin{split}
\| a_{j, K} (t) - g_j(\frac{t}{T_R}, K) \|_{X^l}
& \leq \| b_{j, K} (t) - g_j(\frac{t}{T_R}, K) \|_{X^l} + \| a_{j, K} (t) - b_{j, K} (t) \|_{X^l} \\
& \leq \| b_{j, K} (t) - g_j(\frac{t}{T_R}, K) \|_{X^l} +  C_{\gamma, n} L^\gamma \sum_{d=1}^P \epsilon^{2d} \|u\|^{2d+1}_{X^l_n}  \\
\end{split}
\end{equation*}
for any $\gamma>0$, 
where $g = (g_j)_{j=1}^n$ satisfies
\begin{equation*}
\begin{split}
& -i \partial_t g = \mathcal{T} (g, g, g), \\
& \sup_{0 \leq t \leq M} \|g (t) \|_{X^{l+n+2, N}_n (\mathbb{R}^n)} \leq B , \\
\end{split}
\end{equation*}
with $\mathcal{T} (g, g, g) (K) = (\mathcal{T}_j (g, g, g) (K))_{j=1}^n$, (so $- i\partial_t g_j = \mathcal{T}_j (g, g, g)$) 
\begin{equation*}
\mathcal{T}_j (g, g, g) (K) = \frac{K(j)}{|K|^2} \sum_{m=1}^n K(m) \int_{\mathbb{R}^{3n}} \big(\sum_{l=1}^n g_l (K_1) \overline{g}_l (K_2) \big) g_m (K_3) \delta_{\mathbb{R}^n} (S_3 (K)) \delta_{\mathbb{R}} (\Omega_3 (K)) dK_2 dK_2 dK_3   
\end{equation*}
and 
\begin{equation*}
T_R := \frac{L^{2n}}{\epsilon^2 Z_n (L)} .
\end{equation*}
Initially $a_{j, K} (0) = g_j (0, K)$. We will show by a bootstrapping argument that $\|a_{j, K} (t) \|_{X^l} \leq 2B$ on the time interval $[0, MT_R]$. 

Let us start by assuming that  
\begin{equation} \label{E:pf-th1-eq1}
\begin{split}
\| a_{j, K} (t) - g_j(\frac{t}{T_R}, K) \|_{X^l}
& \leq \| b_{j, K} (t) - g_j(\frac{t}{T_R}, K) \|_{X^l} +  C_{\gamma, n} L^\gamma \sum_{d=1}^P \epsilon^{2d} B^{2d+1} . \
\end{split}
\end{equation}
Denote
\begin{equation*}
w_{j, k}= b_{j, K} (t) - g_j (\frac{t}{T_R}, K) .
\end{equation*}
In order to bound $w_{j, K}$, we write the equation \eqref{vjK-eqn-res} into the form
\begin{equation}
\begin{split}
& -i\partial_t b_{j, K}  = \frac{\epsilon^2}{L^{2n}}   \frac{K(j)}{|K|^2} \sum_{m=1}^n K(m) \sum_{S_3 (K)=0 , \Omega_3 (K) = 0 } (\sum_{l=1}^n a_{l, K_1} \overline{a}_{l, K_2} )  a_{m, K_3} \\
& \qquad \qquad + \sum_{k=1}^n \sum_{d=2}^P \epsilon^{2d}  \sum_{S_{2d+1} (K)=0 }  \frac{\delta \tilde{H}_{2d-1, K}^j}{\delta u_k} \big[ \Delta^{-1} \nabla \nabla \cdot (|a|^2 a)\big]_k  \\
& \qquad \qquad + \epsilon^{2(P+1)} e(-|K|^2 t)  \sum_{k=1}^n \frac{\delta H_{2P+1, K}^j}{\delta u_k} \big[ \Delta^{-1} \nabla \nabla \cdot (|u|^2 u)\big]_k , \\
\end{split}
\end{equation}
where
\begin{equation*}
\begin{split}
& \quad \sum_{k=1}^n \frac{\delta \tilde{H}_{2d-1, K}^j}{\delta u_k} \big[ \Delta^{-1} \nabla \nabla \cdot (|a|^2 a)\big]_k (K) \\
& =  \frac{1}{L^{2nd}} \sum_{S_{2d-1} (K) =0}  \sum_{\substack{J= (j_1, j_2, \cdots , j_{2d-1}) \\ \in \{1, 2, \cdots , n \}^{2d-1}}}  h_{J, 2d-1}^j (K, K_1, \cdots , K_{2d-1}) \\
& \quad \cdot \big( \frac{K_1(j_1)}{|K_1|^2} \sum_{m=1}^n K_1(m) \sum_{ \substack{ N-K_{2d} +K_{2d+1} =K_1 , \\  \Omega_{2d+1} (K) = 0 }} (\sum_{l=1}^n a_{l, N} \overline{a}_{l, K_{2d}} )  a_{m, K_{2d+1}} \big)  \\
& \quad \cdot   \overline{a}_{j_2, K_2} a_{j_3, K_3} \cdots a_{j_{2d-1},K_{2d-1}}  \\
& \quad + \frac{1}{L^{2nd}} \sum_{S_{2d-1} (K) =0}  \sum_{\substack{J= (j_1, j_2, \cdots , j_{2d-1}) \\ \in \{1, 2, \cdots , n \}^{2d-1}}}  h_{J, 2d-1}^j (K, K_1, \cdots , K_{2d-1}) \\
& \quad \cdot \big( \frac{K_2(j_2)}{|K_2|^2} \sum_{m=1}^n K_2(m) \sum_{ \substack{ N-K_{2d} +K_{2d+1} =K_2 , \\  \Omega_{2d+1} (K) = 0 }} (\sum_{l=1}^n \overline{a}_{l, N} a_{l, K_{2d}} )  \overline{a}_{m, K_{2d+1}} \big)  \\
& \quad \cdot  a_{j_1, K_1} a_{j_3, K_3} \cdots a_{j_{2d-1},K_{2d-1}}  \\
& \quad + \cdots \\
\end{split}
\end{equation*}
Here we make use of the constraints $\Omega_3 (K) = \Omega_{2d+1} (K) =0$ ($d=2, 3, \cdots , P$). Hence
\begin{equation} \label{E:pf-th1-eq2}
\begin{split}
-i\partial_t w_{j, K} 
& = \frac{\epsilon^2}{L^{2n}}   \Big\{ \frac{K(j)}{|K|^2} \sum_{m=1}^n K(m) \sum_{S_3 (K)=0 , \Omega_3 (K) = 0 } (\sum_{l=1}^n a_{l, K_1} \overline{a}_{l, K_2} )  a_{m, K_3} - Z_n (L) \mathcal{T}_j (g,g,g)(K) \Big\}\\
& \qquad \qquad + \sum_{k=1}^n \sum_{d=2}^P \epsilon^{2d}  \sum_{S_{2d+1} (K)=0, \Omega_{2d+1} (K) = 0 }  \frac{\delta \tilde{H}_{2d-1, K}^j}{\delta u_k} \big[ \Delta^{-1} \nabla \nabla \cdot (|a|^2 a)\big]_k  \\
& \qquad \qquad + \epsilon^{2(P+1)} e(-|K|^2 t)  \sum_{k=1}^n \frac{\delta H_{2P+1, K}^j}{\delta u_k} \big[ \Delta^{-1} \nabla \nabla \cdot (|u|^2 u)\big]_k  \\
& := I + II + III . \\
\end{split}
\end{equation}
We firstly consider the term $II=  \sum_{k=1}^n \frac{\delta \tilde{H}_{2d-1, K}^j}{\delta u_k} \big[ \Delta^{-1} \nabla \nabla \cdot (|a|^2 a)\big]_k (K) $. By Lemma \ref{Th4-BGHS1}, we have 
\begin{equation*}
\Big| \frac{1}{L^{2n}}  \big( \frac{K_1(j_1)}{|K_1|^2} \sum_{m=1}^n K_1(m) \sum_{N-K_{2d} +K_{2d+1} =K_1 } (\sum_{l=1}^n a_{l, N} \overline{a}_{l, K_{2d}} )  a_{m, K_{2d+1}} \big) \Big| \lesssim_n \frac{Z_n(L)}{L^{2n}} B^3 \la K_1 \ra^{-l} ,
\end{equation*}
\begin{equation*}
\Big| \frac{1}{L^{2n}}  \big( \frac{K_2(j_2)}{|K_2|^2} \sum_{m=1}^n K_2(m) \sum_{N-K_{2d} +K_{2d+1} =K_2 } (\sum_{l=1}^n \overline{a}_{l, N} a_{l, K_{2d}} )  \overline{a}_{m, K_{2d+1}} \big) \Big| \lesssim_n \frac{Z_n(L)}{L^{2n}} B^3 \la K_2 \ra^{-l} ,
\end{equation*}
\begin{equation*}
\cdots
\end{equation*}
for any $j_1$, $j_2$, $\cdots  \in \{1, 2, \cdots , n \}$ and $d$. Hence from Lemma \ref{L:lm13-BGHS1},  
\begin{equation*}
\Big\| \sum_{k=1}^n \frac{\delta \tilde{H}_{2d-1, K}^j}{\delta u_k} \big[ \Delta^{-1} \nabla \nabla \cdot (|a|^2 a)\big]_k (K) \Big\|_{X^l}  \lesssim_n \frac{Z_n(L)}{L^{2n}} B^{2d+1} L^\gamma ,
\end{equation*} 
which gives
\begin{equation}
\|II\|_{X^l} \lesssim_n C_\gamma \big( \sum_{d=2}^P \epsilon^{2d-4} B^{2d+1} \big) \frac{\epsilon^2 Z_n (L)}{L^{2n}} \epsilon^2 L^\gamma . 
\end{equation}
For $III$, by Lemma \ref{L:lm13-BGHS1}, we have
\begin{equation}
\|III\|_{X^l} \lesssim_n C_\gamma B^{2(P+1)} \epsilon^{2(P+1)} L^\gamma .  
\end{equation}
For $I$, we write
\begin{equation}
\begin{split}
& \frac{\epsilon^2}{L^{2n}}   \Big\{ \frac{K(j)}{|K|^2} \sum_{m=1}^n K(m) \sum_{S_3 (K)=0 , \Omega_3 (K) = 0 } (\sum_{l=1}^n a_{l, K_1} \overline{a}_{l, K_2} )  a_{m, K_3} - Z_n (L) \mathcal{T}_j (g,g,g)(K) \Big\} \\
& = \frac{\epsilon^2}{L^{2n}}   \Big\{   \frac{K(j)}{|K|^2} \sum_{m=1}^n K(m) \sum_{S_3 (K)=0 , \Omega_3 (K) = 0 } (\sum_{l=1}^n a_{l, K_1} \overline{a}_{l, K_2} )  a_{m, K_3} \\
& \quad - \frac{K(j)}{|K|^2} \sum_{m=1}^n K(m) \sum_{S_3 (K)=0 , \Omega_3 (K) = 0 } (\sum_{l=1}^n g_{l} (K_1) \overline{g}_{l} (K_2) )  g_{m} (K_3) \Big\} \\
& \quad + \frac{\epsilon^2}{L^{2n}}   \Big\{    \frac{K(j)}{|K|^2} \sum_{m=1}^n K(m) \sum_{S_3 (K)=0 , \Omega_3 (K) = 0 } (\sum_{l=1}^n g_{l} (K_1) \overline{g}_{l} (K_2) )  g_{m} (K_3) - Z_n (L) \mathcal{T}_j (g,g,g)(K) \Big\} \\
& := I_1 + I_2 . \\
\end{split}
\end{equation}
For $I_1$, by Proposition \ref{Th4-BGHS1},
\begin{equation}
\|I_1 \|_{X^l} \lesssim_n C_0 \frac{\epsilon^2 Z_n (L)}{L^{2n}} B^2 \sup_j \| a_{j, K} (t) - g_j (\frac{T}{T_R}, K)\|_{X^l} \leq C_0 \frac{\epsilon^2 Z_n (L)}{L^{2n}} B^2  \| a_{K} (t) - g (\frac{T}{T_R}, K)\|_{X^l} .  
\end{equation}
For the term $I_2$, by Proposition \ref{Th5-BGHS1} item 2),
\begin{equation} \label{E:pf-th1-eq3}
\begin{split}
& \quad \|I_2 \|_{X^l}  \\
& =   \frac{\epsilon^2}{L^{2n}}   \Big\{  \sum_{l=1}^n \frac{K(j)}{|K|^2} \sum_{m=1}^n K(m) \sum_{S_3 (K)=0 , \Omega_3 (K) = 0 }  g_{l} (K_1) \overline{g}_{l} (K_2)   g_{m} (K_3) - \\
& \quad   Z_n (L) \sum_{l=1}^n \frac{K(j)}{|K|^2} \sum_{m=1}^n K(m) \int_{\mathbb{R}^{3n}}   g_l (K_1) \overline{g}_l (K_2)  g_m (K_3) \delta_{\mathbb{R}^n} (S_3 (K)) \delta_{\mathbb{R}} (\Omega_3 (K)) dK_1 dK_2 dK_3    \Big\} \\
& \lesssim_n \frac{\epsilon^2}{L^{2n}}   \sum_{l=1}^n \frac{K(j)}{|K|^2} \sum_{m=1}^n K(m) \Big\{ \sum_{S_3 (K)=0 , \Omega_3 (K) = 0 }  g_{l} (K_1) \overline{g}_{l} (K_2)   g_{m} (K_3) - \\
& \quad   Z_n (L)  \int_{\mathbb{R}^{3n}}   g_l (K_1) \overline{g}_l (K_2)  g_m (K_3) \delta_{\mathbb{R}^n} (S_3 (K)) \delta_{\mathbb{R}} (\Omega_3 (K)) dK_1 dK_2 dK_3      \Big\} \\
& \lesssim_n C B^3 \frac{\epsilon^2 Z_n (L)}{L^{2n}} \delta (L)  .  \\
\end{split}
\end{equation}
Here 
\[   \delta (L) := \begin{cases} 
          ( \log L )^{-1} , \ \text{ when } n =2 , \\
           L^{-1+}  ,  \ \ \ \ \ \ \text{ when } n \geq 3 .
       \end{cases}
\]

Integrating \eqref{E:pf-th1-eq2} in time and combine the result together with \eqref{E:pf-th1-eq1}, we obtain
\begin{equation}
\begin{split}
\Big\| a_{j, K} (t) - g_j(\frac{t}{T_R}, K) \Big\|_{X^l}
& \leq \| b_{j, K} (t) - g_j(\frac{t}{T_R}, K) \|_{X^l} +  C_{\gamma, n} L^\gamma \sum_{d=1}^P \epsilon^{2d} B^{2d+1} \\
& \lesssim_n \int^t_0 \Big\{ C_0 \frac{\epsilon^2 Z_n (L)}{L^{2n}} B^2 \big\| a_{j, K} (s) - g_j(\frac{s}{T_R}, K) \big\|_{X^l} \\
& \quad + C_{\gamma, B} \frac{\epsilon^2 Z_n (L)}{L^{2n}} \delta (L)  + C_{\gamma, B} \frac{\epsilon^2 Z_n (L)}{L^{2n}} \epsilon^2 L^\gamma + C_{\gamma, B}  \epsilon^{2P+1} L^\gamma   \Big\}  ds \\
& \quad +  C_{\gamma, B} L^\gamma  \epsilon^2 . \\
\end{split}
\end{equation}
By Gronwall’s inequality and $0 \leq t \leq MT_R $, we obtain,
\begin{equation}
\begin{split}
\Big\| a_{j, K} (t) - g_j(\frac{t}{T_R}, K) \Big\|_{X^l}
& \leq  C_{\gamma, B, n} (\epsilon^2 L^\gamma + M \delta (L) + M \epsilon^2 L^\gamma + M \epsilon^{2P+1} \L^{2+\gamma} ) e^{C_0 B^2 M} .
\end{split}
\end{equation}
Thus, choosing $L$ large, $\epsilon^2 L^\gamma$ small , and $P$ large, we arrive at
\begin{equation}
\begin{split}
\sup_{0 \leq t \leq T_R \delta_0} \Big\| a_{j, K} (t) - g_j(\frac{t}{T_R}, K) \Big\|_{X^l}
& \leq  C_{\gamma, B, n} (\epsilon^2 L^\gamma +  \delta (L)  ) < \frac{B}{2} ,
\end{split}
\end{equation}
and hence 
\begin{equation}
\|a_K (t) \|_{X^l} \leq 2B .
\end{equation}
The bootstrap argument is then complete, which gives the desired Theorem \ref{th1}. 

Theorem \ref{th2} is proved in a similar way as the one for Theorem \ref{th1}. One only needs to replace the term 
$$I = \frac{\epsilon^2}{L^{2n}}   \Big\{ \frac{K(j)}{|K|^2} \sum_{m=1}^n K(m) \sum_{S_3 (K)=0 , \Omega_3 (K) = 0 } (\sum_{l=1}^n a_{l, K_1} \overline{a}_{l, K_2} )  a_{m, K_3} - Z_n (L) \mathcal{T}_j (g,g,g)(K) \Big\} $$
in \eqref{E:pf-th1-eq2} by 
\begin{equation*}
\begin{split}
\tilde{I} =
& \frac{\epsilon^2}{L^{2n}}   \Big\{ \frac{K(j)}{|K|^2} \sum_{m=1}^n K(m) \sum_{S_3 (K)=0 , \Omega_3 (K) = 0 } (\sum_{l=1}^n a_{l, K_1} \overline{a}_{l, K_2} )  a_{m, K_3}  \\
& - Z_n (L) [ \mathcal{T}_j (g,g,g)(K) +  \frac{\zeta(2)}{\log L}  \mathcal{C} (g) (K) ] \Big\} 
\end{split}
\end{equation*}
and then make use of \eqref{Th5-BGHS1-eq2} to estimate it.

\end{proof}

%

The proof of Theorem \ref{th1-NLS} and Theorem \ref{th2-NLS} can be done in a similar (and simpler) way as in the proof of Theorem \ref{th1} and Theorem \ref{th2} so we omit them. In fact the process is close to the derivation of the continuous resonant equation for the one component NLS equation shown in \cite{BGHS1}.

\section{The General Case $p \in \mathbb{N}$}
\label{S:General-case-p}

In this section, we consider a generalized version of \eqref{VNLS}
\begin{equation} \label{VNLS-p}
-i\partial_t E + \frac{1}{2 \pi} \Delta E = \Delta^{-1} \nabla \nabla \cdot (|E|^{2p} E) , \ E|_{t=0} = \epsilon E_0 , \ (t, x) \in  \mathbb{R} \times \mathbb{T}^n_L  .
\end{equation}
Here $p \in \mathbb{N}$, $E = E(x, t) : \mathbb{T}^n_L \times \mathbb{R} \rightarrow \mathbb{C}^n$ is an \emph{irrotational} complex-valued vector field (see \cite{CW1}), and $E_0 = E_0(x) : \mathbb{T}^n_L  \rightarrow \mathbb{C}^n$ is a fixed \emph{irrotational} complex-valued vector field. Using the ansatz $E= \epsilon u$, then the equation for $u$ is
\begin{equation} \label{VNLS-p-rescaled}
-i\partial_t u + \frac{1}{2 \pi} \Delta u = \epsilon^{2p} \Delta^{-1} \nabla \nabla \cdot (|u|^{2p} u) , \ u|_{t=0} =   E_0 ,  \ (t, x) \in  \mathbb{R} \times \mathbb{T}^n_L  .
\end{equation}
Define $\mathcal{T}^{p, n} = (\mathcal{T}_j^{p, n} )_{j=1, 2, \cdots , n} $, 
\begin{equation*}
\begin{split}
\mathcal{T}_j^{p, n} (g, g, g) (K) 
& := \frac{K(j)}{|K|^2} \sum_{m=1}^n K(m) \int_{\mathbb{R}^{n(2p+1)}} \big(\sum_{l=1}^n g_{l} (K_1)  \overline{g}_{l} (  K_2) \big) \\
& \quad \cdots \big(\sum_{l=1}^n g_{l} (  K_{2p-1})  \overline{g}_{l} (  K_{2p} ) \big) \cdot g_{m} ( K_{2p+1}) \\
& \quad \cdot   \delta_{\mathbb{R}^n} (S_{2p+1} (K)) \delta_{\mathbb{R}} (\Omega_{2p+1} (K)) dK_1 dK_2 dK_3 \cdots dK_{2p+1}  ,  \\
\end{split}
\end{equation*}
We have

\begin{theorem} \label{th3}
Fix $l>2n$ and $0 < \gamma <1$, $pn \neq 2$. Let $g_0 \in X^{l+n+2, 3n+3}_n (\mathbb{R}^n)$, and suppose that $g(t, \xi)$ is a solution of 
\begin{equation} \label{CR-p}
-i\partial_t g(t,\xi)= \mathcal{T}^{p, n}   (g(t, \cdot) , g(t, \cdot) , g(t, \cdot)) (t, \xi) , \ \xi \in \mathbb{R}^n .
\end{equation}
over a time interval $[0, M]$ with initial data $g_0 = g(t=0)$. Denote
\begin{equation*}
B := \sup_{t \in [0, M] } \|g(t)\|_{ X^{l+n+2, 3n+3}_n (\mathbb{R}^n) } .
\end{equation*}
Let $u$ be a solution of \eqref{VNLS-p-rescaled} with initial data $u_0 = \frac{1}{L^n} \sum_{\mathbb{Z}_L^n} g_0 (K) e(K \cdot x)$, and set for $K \in \mathbb{Z}_L^n$
\begin{equation*}
a_K (t) := u_K (t) e(-|K|^2 t) .
\end{equation*} 
Then for $L$ sufficiently large and $\epsilon^2 L^\gamma$ sufficiently small, there exists a constant $C_{\gamma, M , B}$, such that for all $t \in [0, MT_R]$,
\begin{equation}
\Big\| a_K(t) - g(\frac{t}{T_R}, K)  \Big\|_{X^l_n (\mathbb{Z}_L^n)} \lesssim_{n, p} C_{\gamma, M, B} (\delta (L) + \epsilon^2 L^\gamma) ,
\end{equation}
where
\begin{equation}
T_R := \frac{L^{2n}}{\epsilon^2 Z_n (L)} .
\end{equation}
\end{theorem}

\begin{theorem} \label{th4}
Let $pn=2$. Fix $l>4$ and $0 < \gamma <1$. Let $g_0 \in X^{l+6, 15}_2 (\mathbb{R}^2)$, and suppose that $g(t, \xi)$ is a solution of 
\begin{equation} \label{CR-p-np=2}
-i\partial_t g(t,\xi)= \mathcal{T}^{p, n}   (g(t, \cdot) , g(t, \cdot) , g(t, \cdot)) (t, \xi) + \frac{\zeta(2)}{\log L} \mathcal{C} (g(t, \cdot)) (t, \xi) , \ \xi \in \mathbb{R}^2 .
\end{equation}
over a time interval $[0, M]$ with initial data $g_0 = g(t=0)$. Here $\mathcal{C} (g(t, \cdot))$ is a correction term defined in \eqref{Th6-BGHS1-eq1}. Denote
\begin{equation*}
B := \sup_{t \in [0, M] } \|g(t)\|_{ X^{l+6, 15}_2 (\mathbb{R}^2) } .
\end{equation*}
Let $u$ be a solution of \eqref{VNLS-p-rescaled} with initial data $u_0 = \frac{1}{L^2} \sum_{\mathbb{Z}_L^2} g_0 (K) e(K \cdot x)$, and set for $K \in \mathbb{Z}_L^2$
\begin{equation*}
a_K (t) := u_K (t) e(-|K|^2 t) .
\end{equation*} 
Then for $L$ sufficiently large and $\epsilon^2 L^\gamma$ sufficiently small, there exists a constant $C_{\gamma, M , B}$, such that for all $t \in [0, MT_R]$,
\begin{equation}
\Big\| a_K(t) - g(\frac{t}{T_R}, K)  \Big\|_{X^l_2 (\mathbb{Z}_L^2)} \lesssim_p C_{\gamma, M, B} ( \frac{1}{L^{1/3}-\gamma} + \epsilon^2 L^\gamma) ,
\end{equation}
where
\begin{equation}
T_R := \frac{L^{4}}{\epsilon^2 Z_2 (L)} = \frac{\zeta (2) L^2}{\epsilon^2 \log L} .
\end{equation}
\end{theorem}

The proofs for Theorem \ref{th3} and Theorem \ref{th4} are almost the same as the ones for Theorem \ref{th1} and Theorem \ref{th2}. It suffices to replace Proposition \ref{Th5-BGHS1} by Proposition \ref{Th6-BGHS1}. 
 

In a similar manner, we also consider a generalized version of the coupled NLS equation \eqref{NLS}:
\begin{equation} \label{NLS-p}
-i\partial_t E + \frac{1}{2 \pi} \Delta E =   |E|^{2p} E , \ E|_{t=0} = \epsilon E_0 , \ (t, x) \in  \mathbb{R} \times \mathbb{T}^n_L  .
\end{equation}
Here $p \in \mathbb{N}$, $E = E(x, t) : \mathbb{T}^n_L \times \mathbb{R} \rightarrow \mathbb{C}^d$ being a complex-valued vector field, and $E_0 = E_0(x) : \mathbb{T}^n_L  \rightarrow \mathbb{C}^d$ is a fixed complex-valued vector field. Using the ansatz $E= \epsilon u$, then the equation for $u$ is
\begin{equation} \label{NLS-p-rescaled}
-i\partial_t u + \frac{1}{2 \pi} \Delta u = \epsilon^{2p}   |u|^{2p} u , \ u|_{t=0} =   E_0 ,  \ (t, x) \in  \mathbb{R} \times \mathbb{T}^n_L  .
\end{equation}
Define $\mathcal{T}^{p, n} = (\mathcal{T}_j^{p, n} )_{j=1, 2, \cdots , d} $, 
\begin{equation*}
\begin{split}
\mathcal{T}_j^{p, n} (g, g, g) (K) 
& :=  \int_{\mathbb{R}^{n(2p+1)}} \big(\sum_{l=1}^d g_{l} (K_1)  \overline{g}_{l} (  K_2) \big)  \cdots \big(\sum_{l=1}^d g_{l} (  K_{2p-1})  \overline{g}_{l} (  K_{2p} ) \big) \cdot g_{m} ( K_{2p+1}) \\
& \quad \cdot   \delta_{\mathbb{R}^n} (S_{2p+1} (K)) \delta_{\mathbb{R}} (\Omega_{2p+1} (K)) dK_1 dK_2 dK_3 \cdots dK_{2p+1}  ,  \\
\end{split}
\end{equation*}
We can obtain, using the same idea as in the proof of Theorem \ref{th1} and \ref{th2}:

\begin{theorem} \label{th6}
Fix $l>2n$ and $0 < \gamma <1$, $pn \neq 2$. Let $g_0 \in X^{l+n+2, 3n+3}_n (\mathbb{R}^n)$, and suppose that $g(t, \xi)$ is a solution of 
\begin{equation} \label{CR-NLS-p}
-i\partial_t g(t,\xi)= \mathcal{T}^{p, n}   (g(t, \cdot) , g(t, \cdot) , g(t, \cdot)) (t, \xi) , \ \xi \in \mathbb{R}^n .
\end{equation}
over a time interval $[0, M]$ with initial data $g_0 = g(t=0)$. Denote
\begin{equation*}
B := \sup_{t \in [0, M] } \|g(t)\|_{ X^{l+n+2, 3n+3}_d (\mathbb{R}^n) } .
\end{equation*}
Let $u$ be a solution of \eqref{NLS-p-rescaled} with initial data $u_0 = \frac{1}{L^n} \sum_{\mathbb{Z}_L^n} g_0 (K) e(K \cdot x)$, and set for $K \in \mathbb{Z}_L^n$
\begin{equation*}
a_K (t) := u_K (t) e(-|K|^2 t) .
\end{equation*} 
Then for $L$ sufficiently large and $\epsilon^2 L^\gamma$ sufficiently small, there exists a constant $C_{\gamma, M , B}$, such that for all $t \in [0, MT_R]$,
\begin{equation}
\Big\| a_K(t) - g(\frac{t}{T_R}, K)  \Big\|_{X^l_d (\mathbb{Z}_L^n)} \lesssim_{n, d, p} C_{\gamma, M, B} (\delta (L) + \epsilon^2 L^\gamma) ,
\end{equation}
where
\begin{equation}
T_R := \frac{L^{2n}}{\epsilon^2 Z_n (L)} .
\end{equation}
\end{theorem}

\begin{theorem} \label{th7}
Let $pn=2$. Fix $l>4$ and $0 < \gamma <1$. Let $g_0 \in X^{l+6, 15}_d (\mathbb{R}^2)$, and suppose that $g(t, \xi)$ is a solution of 
\begin{equation} \label{CR-NLS-p-np=2}
-i\partial_t g(t,\xi)= \mathcal{T}^{p, n}   (g(t, \cdot) , g(t, \cdot) , g(t, \cdot)) (t, \xi) + \frac{\zeta(2)}{\log L} \mathcal{C} (g(t, \cdot)) (t, \xi) , \ \xi \in \mathbb{R}^2 .
\end{equation}
over a time interval $[0, M]$ with initial data $g_0 = g(t=0)$. Here $\mathcal{C} (g(t, \cdot))$ is a correction term defined in \eqref{Th6-BGHS1-eq1}. Denote
\begin{equation*}
B := \sup_{t \in [0, M] } \|g(t)\|_{ X^{l+6, 15}_d (\mathbb{R}^2) } .
\end{equation*}
Let $u$ be a solution of \eqref{NLS-p-rescaled} with initial data $u_0 = \frac{1}{L^2} \sum_{\mathbb{Z}_L^2} g_0 (K) e(K \cdot x)$, and set for $K \in \mathbb{Z}_L^2$
\begin{equation*}
a_K (t) := u_K (t) e(-|K|^2 t) .
\end{equation*} 
Then for $L$ sufficiently large and $\epsilon^2 L^\gamma$ sufficiently small, there exists a constant $C_{\gamma, M , B}$, such that for all $t \in [0, MT_R]$,
\begin{equation}
\Big\| a_K(t) - g(\frac{t}{T_R}, K)  \Big\|_{X^l_d (\mathbb{Z}_L^2)} \lesssim_{d, p} C_{\gamma, M, B} ( \frac{1}{L^{1/3}-\gamma} + \epsilon^2 L^\gamma) ,
\end{equation}
where
\begin{equation}
T_R := \frac{L^{4}}{\epsilon^2 Z_2 (L)} = \frac{\zeta (2) L^2}{\epsilon^2 \log L} .
\end{equation}
\end{theorem}


\section{Analysis for \eqref{CR}}
\label{S:CR-analysis}

\subsection{Hamiltonian Structure for \eqref{CR}}
\label{S:HamilStructure}

Let $n \geq 3$, $\mathcal{T} = (\mathcal{T}_j )_{j=1, 2, \cdots , n} $, 
\begin{equation*}
\mathcal{T}_j (f, g, h) (K) = \frac{K(j)}{|K|^2} \sum_{m=1}^n K(m) \int_{\mathbb{R}^{3n}} \big(\sum_{l=1}^n f_l (K_1) \overline{g}_l (K_2) \big) h_m (K_3) \delta_{\mathbb{R}^n} (S_3 (K)) \delta_{\mathbb{R}} (\Omega_3 (K)) dK_1 dK_2 dK_3   . 
\end{equation*} 
Define
\begin{equation} \label{D:E-VNLS}
\begin{split}
\mathcal{E} (f, g, h, p)  
& := \int_{\mathbb{R}^{4n}}    \big(\sum_{l=1}^n f_l (K_1) \overline{g}_l (K_2) \big) \big(\sum_{m=1}^n   h_m (K_3) \overline{p}_m (K) \big)  \\
& \quad \cdot  \delta_{\mathbb{R}^n} (S_3 (K)) \delta_{\mathbb{R}} (\Omega_3 (K)) dK_1 dK_2 dK_3  dK . \\
\end{split}
\end{equation} 
then the (CR) equation \eqref{CR} can be derived from the Hamiltonian
\begin{equation}
\mathcal{E} (g) := \mathcal{E} (g, g, g,g) .
\end{equation}
In other words, \eqref{CR} can also be written as 
\begin{equation}
\partial_t g_j = \frac{1}{2} J_j \nabla_{\overline{g}} \mathcal{E} (g) , \ j = 1, 2, \cdots , n 
\end{equation}
where
\begin{equation}
J_j = i \frac{K(j)}{|K|^2}  K \cdot 
\end{equation}
and hence
\begin{equation}
\mathcal{T}_j (g, g, g) (K) = \frac{1}{2} \frac{K(j)}{|K|^2}  ( K \cdot \nabla_{\overline{g}} \mathcal{E} (g) )  . 
\end{equation}



From now on, we use $\check{f}$ to denote the inverse Fourier transform of $f$, and denote the harmonic oscillator $H= - \Delta + |x|^2$, and $\Pi_n$ be the projector on the $n$-th eigenspace of $H$.

\begin{lemma} \label{L:lm2.2-GHT1}
The quantity $\mathcal{E}$ satisfies
\begin{equation} \label{E:lm2.2-GHT1-eq1}
\mathcal{E} (f, g, h, p) = (2\pi)^{n-1}  \int_{\mathbb{R}} \int_{\mathbb{R}^n}  \big[ \sum_{l=1}^n e^{it \Delta} \check{f}_l   \overline{e^{it \Delta} \check{g}_l  } \big]  \big[ \sum_{m =1}^n e^{it \Delta} \check{h}_m  \overline{e^{it \Delta}   \check{p}_m } \big] (t, x) dx dt . 
\end{equation}
\begin{equation} \label{E:lm2.2-GHT1-eq2}
\mathcal{E} (f, g, h, p) = (2 \pi)^{n-1}   \int^{\pi/4}_{-\pi/4} \int_{\mathbb{R}^n}  \big[ \sum_{l=1}^n e^{-itH} \check{f}_l \overline{e^{-itH} \check{g}_l} \big] \big[ \sum_{m =1}^n e^{-itH} \check{h}_m \overline{ e^{-itH}    \check{p}_m  } \big] (t, x) dx dt  . 
\end{equation}
Consequently,  
\begin{equation} \label{E:lm2.2-GHT1-eq3}
\mathcal{T}_j (f, g, h) =  (2\pi)^{n-1}   \frac{K(j)}{|K|^2} \mathcal{F} \Big\{ \int_{\mathbb{R}}  \sum_{m =1}^n  K(m) e^{-it \Delta}   \big[   \big( \sum_{l=1}^n  e^{it \Delta} \check{f}_l   \overline{e^{it \Delta} \check{g}_l  } \big) e^{it \Delta} \check{h}_m   \big] (t, x)   dt \Big\}  . 
\end{equation}
\begin{equation} \label{E:lm2.2-GHT1-eq4}
\mathcal{T}_j (f, g, h) =  (2 \pi)^{n-1}   \frac{K(j)}{|K|^2} \mathcal{F} \Big\{  \int^{\pi/4}_{-\pi/4}   \sum_{m =1}^n  K(m) e^{it H} \big[  \big( \sum_{l=1}^n e^{-itH} \check{f}_l \overline{e^{-itH} \check{g}_l} \big) e^{-itH} \check{h}_m  \big]   (t, x) dt  \Big\}  . 
\end{equation}
\end{lemma}

\begin{proof}
\begin{equation}
\begin{split}
& \quad \mathcal{E} (f, g, h, p)   \\
& =  \int_{\mathbb{R}^{4n}}        \big(\sum_{l=1}^n f_l (K_1) \overline{g}_l (K_2) \big)  \big(\sum_{m=1}^n  h_m (K_3) \overline{p}_m (K)  \big)  \delta_{\mathbb{R}^n} (S_3 (K)) \delta_{\mathbb{R}} (\Omega_3 (K)) dK_1 dK_2 dK_3 dK   \\
& =  (2 \pi)^{-1}  \int_{\mathbb{R}}  \int_{\mathbb{R}^{4n}} e^{- it \Omega_3 (K)}      \big(\sum_{l=1}^n f_l (K_1) \overline{g}_l (K_2) \big) \big(\sum_{m=1}^n  h_m (K_3) \overline{p}_m (K)  \big)  \\
& \quad \cdot  \delta_{\mathbb{R}^n} (S_3 (K))  dK_1 dK_2 dK_3 dK dt   \\
& =  (2 \pi)^{-1-n} \int_{\mathbb{R}}  \int_{\mathbb{R}^{n}} \int_{\mathbb{R}^{4n}} e^{ ix S_3 (K)}  e^{- it \Omega_3 (K)}     \\
& \quad \cdot   \big(\sum_{l=1}^n f_l (K_1) \overline{g}_l (K_2) \big)  \big(\sum_{m=1}^n  h_m (K_3) \overline{p}_m (K) \big)   dK_1 dK_2 dK_3 dK dx dt  \\
& =  (2 \pi)^{n-1} \int_{\mathbb{R}}  \int_{\mathbb{R}^{n}}  \big[ \sum_{l=1}^n e^{it \Delta} \check{f}_l (x) \overline{e^{it \Delta} \check{g}_l (x)} \big]  \big[ \sum_{m=1}^n  e^{it \Delta} \check{h}_m (x) \overline{e^{it \Delta}  \check{p}_m (x)} \big] dx dt .  \\
\end{split}
\end{equation}
This gives \eqref{E:lm2.2-GHT1-eq1}.

Now, let $\check{f}_H := e^{-itH \check{f}}$, $\check{f}_\Delta := e^{it\Delta \check{f}}$, and similarly for $\check{g}$, $\check{h}$, $\check{p}$. Then the lens transform gives
\begin{equation}
\check{f}_\Delta (t, x) = \frac{1}{\sqrt{1+4t^2}} \check{f}_H (\frac{\arctan (2t)}{2} , \frac{x}{\sqrt{1+4t^2}} ) e^{\frac{i |x|^2 t}{1+4t^2}} .
\end{equation}
We make the change of variables $y= \frac{x}{\sqrt{1+4t^2}}$ in the equation \eqref{E:lm2.2-GHT1-eq1}, and let $\tau = \frac{\arctan (2t)}{2}$:
\begin{equation}
\begin{split}
\mathcal{E} (f, g, h, p)   
& = (2 \pi)^{n-1}  \int_{\mathbb{R}} \int_{\mathbb{R}^n}  \big[ \sum_{l=1}^n \check{f}_{H, l} \overline{\check{g}_{H, l}} \big] \big[ \sum_{m=1}^n  \check{h}_{H, m} \overline{ \check{p}_{H, m} } \big] (\frac{\arctan (2t)}{2}, y) dy dt \\
& = (2 \pi)^{n-1}  \int^{\pi/4}_{-\pi/4} \int_{\mathbb{R}^n}  \big[ \sum_{l=1}^n \check{f}_{H, l} \overline{\check{g}_{H, l}} \big]  \big[ \sum_{m=1}^n \check{h}_{H, m} \overline{ \check{p}_{H, m}} \big]  (\tau, y) dy d\tau . \\
\end{split}
\end{equation} 
This gives \eqref{E:lm2.2-GHT1-eq2}. The relations for $\mathcal{T}$ are obtained using $  \mathcal{T}_j (f, g, h) =  \frac{1}{2} J_j \nabla_{\overline{p}} \mathcal{E} (f, g, h, p)$, $\mathcal{T}_j (g, g, g) (K) = \frac{1}{2} \frac{K(j)}{|K|^2}  ( K \cdot \nabla_{\overline{g}} \mathcal{E} (g) ) $.
\end{proof}

\begin{lemma} \label{L:lm2.3-GHT1}
The quantity $\mathcal{E}$ satisfies
\begin{equation} \label{E:lm2.3-GHT1-eq1}
\begin{split}
& \mathcal{E} (f, g, h, p) =  \\
& \frac{(2\pi)^{n}}{4}  \sum_{\substack{\ell_1 , \ell_2 , \ell_3 , \ell_4 \in \mathbb{Z}_{\geq 0} \\ \ell_1 + \ell_3 = \ell_2 + \ell_4}}  \int_{\mathbb{R}^n}  \big[ \sum_{l=1}^n  \Pi_{\ell_1} \check{f}_l (x)  \overline{\Pi_{\ell_2}  \check{g}_l (x) } \big] \big[ \sum_{m =1}^n   \Pi_{\ell_3}  \check{h}_m (x) \overline{\Pi_{\ell_4}   \check{p}_m (x)} \big] dx  . 
\end{split}
\end{equation}
Hence
\begin{equation} \label{E:lm2.3-GHT1-eq2}
\begin{split}
& \mathcal{T}_j (f, g, h) =  \\
& \frac{(2\pi)^{n}}{4}   \frac{K(j)}{|K|^2}  \sum_{m =1}^n  K(m)   \sum_{\substack{\ell_1 , \ell_2 , \ell_3 , \ell_4 \in \mathbb{Z}_{\geq 0} \\ \ell_1 + \ell_3 = \ell_2 + \ell_4}} \mathcal{F} \Big\{  \Pi_{\ell_4}   \Big[  \big[ \sum_{l=1}^n  \Pi_{\ell_1} \check{f}_l  (x)  \overline{ \Pi_{\ell_2}  \check{g}_l (x) } \big] \Pi_{\ell_3}  \check{h}_m (x) \Big]   \Big\}  . 
\end{split}
\end{equation}
\end{lemma}

\begin{proof}
We compute $\mathcal{E}_j$ using the expression \eqref{E:lm2.2-GHT1-eq2} for the eigenfunctions of $H$. Therefore we assume that $\Pi_{\ell_1} f = f $, $\Pi_{\ell_2} g = g $, $\Pi_{\ell_3} h = h $, $\Pi_{\ell_4} p = p $. Then
\begin{equation}  
\begin{split}
\mathcal{E} (f, g, h, p) 
& = (2 \pi)^{n-1} \sum_{\substack{\ell_1 , \ell_2 , \ell_3 , \ell_4 \in \mathbb{R}_+ \\ \ell_1 + \ell_3 = \ell_2 + \ell_4}}   \int^{\pi/4}_{-\pi/4}  e^{-2i (\ell_1 - \ell_2 + \ell_3 -\ell_4 )t} dt  \\
& \quad \cdot \int_{\mathbb{R}^n}  \big[ \sum_{l=1}^n   \check{f}_l \overline{ \check{g}_l} \big]  \big[ \sum_{m =1}^n \check{h}_m \overline{  \check{p}_m )}  \big]  (t, x)dx  .  \\
\end{split} 
\end{equation}
Now we use that $\Pi_{\ell_1} f(-x) = (-1)^{\ell_1} f(x)$, etc., thus $\int_{\mathbb{R}^n}      \check{f}_l \overline{ \check{g}_l}    \check{h}_m \overline{   ( \Delta^{-1} \partial_{x_j}   \partial_{x_m} \check{p}_j )} (t, x) dx =0$ unless $\ell_1 - \ell_2 + \ell_3 -\ell_4 $ is even. Notice that $\ell_1 - \ell_2 + \ell_3 -\ell_4 $ being even is equivalent to $\int^{\pi/4}_{-\pi/4}  e^{-2i (\ell_1 - \ell_2 + \ell_3 -\ell_4 )t} dt = \frac{\pi}{2} \delta (\ell_1 - \ell_2 + \ell_3 -\ell_4)$. \eqref{E:lm2.3-GHT1-eq1} then follows.
\end{proof}

\subsection{Symmetries of $\mathcal{T}$ and Conservation Laws for \eqref{CR}}
\label{S:SymConservation} 
 
We observe
 
%

\begin{lemma}   \label{L:symmetry}
The following symmetries leave the Hamiltonian $\mathcal{E}$ invariant: \\
1) Rotation: $g   \mapsto e^{i \theta_0} g $ for $\theta_0  \in \mathbb{R}$; \\
2) Modulation: $g \mapsto e^{i K \cdot x_0} g$ for all $x_0 \in \mathbb{R}^n$; \\
3) Quadratic modulation: $g \mapsto e^{i \tau |K|^2} g$ for all $\tau \in \mathbb{R}$; \\
4) Rotation: $g \mapsto O^T g(O \cdot ) $ for all any $O$ in the orthogonal group $O(n)$; \\
5) Scaling: $g \mapsto \lambda^{\frac{3n-2}{4}} g (\lambda \cdot ) $ for all $\lambda \in \mathbb{R} \setminus \{ 0 \}$. \\
\end{lemma}

\begin{proof}
These properties can be found by inspection. In particular, we notice that under these symmetries, the irrotationality of $\check{g}$ is kept unchanged. That is, under these symmetries, the property that $\frac{K(j)}{|K|^2} K \cdot g(K) =g(K)$, $g (K) = K G (K)$ for some scalar function $ G (K) $ always holds.
\end{proof}

Hence by Noether’s theorem, we have the following conserved quantities associated to the symmetries 1) -- 4), respectively:

\begin{lemma} \label{L:conservation-law}
The following quantities are conserved by the flow of (CR):  \\
1) Rotation: 
\begin{equation*}
\begin{split}
& M = \sum_{j=1}^n  \int |g_j|^2 dK = \int |g|^2 dK  .  \\ 
\end{split}
\end{equation*}
2) Modulation: 
\begin{equation*}
\begin{split}
& \int K (j) |g|^2 dK , \  j = 1, 2, \cdots , n .  \\ 
\end{split}
\end{equation*}
3) Quadratic modulation: 
\begin{equation*}
\begin{split}
& \int |K|^2 |g|^2 dK . \\
\end{split}
\end{equation*}
4) Rotation: 
\begin{equation*}
\begin{split}
& \Im  \int ( K(j) \partial_{K(k)} - K(k) \partial_{K(j)} ) g (K) \cdot \overline{g (K)} dK , \ j \neq k ,  \ j , k = 1, \cdots , n .  \\
\end{split}
\end{equation*}
\end{lemma}


\begin{proof}
The conservation laws can be obtained by observing the Lie groups associated with the symmetries in Lemma \ref{L:symmetry}. In particular, we notice that $\check{g}$ is curl free, $\frac{K(j)}{|K|^2} K \cdot g(K) =g_j (K)$.



For 4) in Lemma \ref{L:symmetry}, the associated conserved quantities are given by
\begin{equation*}
\frac{1}{2}  \Im \sum_{j=1}^n \int ( \nabla_K g_j (K) \cdot A K ) \overline{g_j (K)}  dK  ,
\end{equation*}
where $A$ is any $n \times n$ skew-orthogonal matrix (i.e. $A^\top = -A$). Taking $A=  e_{jk}^n - e_{kj}^n$ gives the conserved quantity $\Im  \int ( K(j) \partial_{K(k)} - K(k) \partial_{K(j)} ) g (K) \cdot \overline{g (K)} dK$. 

\end{proof}

\subsection{Well-Posedness for \eqref{CR}}
\label{S:LWP}

We define
\begin{equation}
\begin{split}
& \|f \|_{L^{p, l}  } :=  \| \la K \ra^l f (K) \|_{L^p  }  , \  L^{p, l}_n := (L^{p, l})^n  ,  \\
& \|f \|_{\dot{L}^{p, l}  } :=  \| | K |^l f (K) \|_{L^p  }  , \  \dot{L}^{p, l}_n := (\dot{L}^{p, l})^n  ,  \\
& \|f \|_{W^{p, l}  } :=  \| \la D \ra^l f (K) \|_{L^p  }  , \  W^{p, l}_n := (W^{p, l})^n , \\
\end{split}
\end{equation}
and recall
\begin{equation}
\begin{split}
& \|f \|_{X^{l, N}  } := \sum_{0 \leq |\alpha| \leq N} \| \nabla^\alpha  f \|_{X^l  }  = \sum_{0 \leq |\alpha| \leq N} \| \nabla^\alpha  f \|_{L^{\infty, l} }     ,   \   X^{l, N}_n  := (X^{l, N}  )^n  . \\
\end{split}
\end{equation}

We have the following proposition for the boundedness property of the operator $\mathcal{T} = (\mathcal{T}_j )_{j=1, 2, \cdots , n} $,
\begin{equation}  \label{def-T-VNLS}
\mathcal{T}_j (g, g, g) (K) := \frac{K(j)}{|K|^2} \sum_{m=1}^n K(m) \int_{\mathbb{R}^{3n}} \big(\sum_{l=1}^n g_l (K_1) \overline{g}_l (K_2) \big) g_m (K_3) \delta_{\mathbb{R}^n} (S_3 (K)) \delta_{\mathbb{R}} (\Omega_3 (K)) dK_1 dK_2 dK_3  .
\end{equation} 

\begin{proposition} \label{L:BGHS2-prop5-VNLS}
The trilinear operator $\mathcal{T}$ is bounded from $X \times X \times X$ to $X$ for the following Banach spaces $X$:  \\
1) $X= \dot{L}^{p, \frac{n-2}{2}}_n $;  \\
2) $X= L^{p, l}_n $, $l \geq \frac{n-2}{2}$;  \\
3) $X= L^{\infty, l}_n $, $l  >  n-1$;  \\
4) $X= L^{p, l}_n $, $p \geq 2$, $l  >  n-1- \frac{n}{p}$;  \\
5) $X= X^{l, N}_n $, $l  >  n-1$, $N \geq 0$.  \\
\end{proposition}

\begin{remark}
The borderline spaces for well-posedness above are $ \dot{L}^{p, n-1-\frac{n}{p}}_n$. They share the same scaling, and are also scale-invariant for the cubic NLS in $n$-dimension \eqref{NLS} (when viewed as spaces for $\hat{u}$).
\end{remark}

\begin{proof}
The proof is very similar to the one for Proposition 5 in \cite{BGHS2} so we omit it. Indeed, we can perform an argument which is essentially the same as in Proposition 5 in \cite{BGHS2} for each $\mathcal{T}_j$. Notice that $ \frac{K(j)  K(m) }{|K|^2}  \leq 1 $.
\end{proof}

From Proposition \ref{L:BGHS2-prop5-VNLS}, we conclude the following local well-posedness results:

\begin{theorem}
1) For $n \geq 3$, $X$ any spaces given in Proposition \ref{L:BGHS2-prop5-VNLS}, the equation \eqref{CR-NLS} is locally well-posed in $X$;  \\
2) For $n=2$, $X$ any spaces given in Proposition \ref{L:BGHS2-prop5-VNLS}, the equation $-i\partial_t g(t,\xi)= \mathcal{T} (g(t, \cdot) , g(t, \cdot) , g(t, \cdot)) (t, \xi) $ with $\mathcal{T}$ in \eqref{def-T-VNLS} is locally well-posed in $X$.  \\
\end{theorem}

\begin{remark}
For the well-posedness of equation \eqref{CR-NLS-n=2} in the case $n=2$, we also need to consider the last term on the right hand side in \eqref{CR-NLS-n=2}, which, for the time being, lack of understanding and sufficient estimate for local well-posedness.
\end{remark}

\subsection{Stationary Waves of \eqref{CR}}
\label{S:StationaryWave}

In this section, we consider the existence of solutions of the CR equation \eqref{CR} of the vector NLS of the type
\begin{equation}
g(t, K) = e^{-i(\mu + \lambda |K|^2 + \nu \cdot K)t } \psi (K) ,
\end{equation}
where $\lambda$, $\mu \in \mathbb{R}$, $\nu \in \mathbb{R}^n$. For $g$ to solve \eqref{CR}, it suffices that $\psi$ solves
\begin{equation}
(\mu + \lambda |K|^2 + \nu \cdot K) \psi = \mathcal{T} (\psi, \psi, \psi ) 
\end{equation}
where $\mathcal{T}$ is as given in \eqref{def-T-VNLS}.
Notice that $g$ defined above oscillates in Fourier space, but it actually travels in physical space, as can be seen by taking its inverse Fourier transform:
\begin{equation}
\check{g} (t, x) = e^{-i(\mu - \lambda \Delta )t } \check{\psi} (x- \nu t) ,
\end{equation}
The conservation of position gives a restriction on the relation between $\nu$ and $\lambda$. Using the identity $[x, e^{it \Delta} ] = -2 it \nabla e^{it \Delta} $, we have
\begin{equation*}
\int_{\mathbb{R}^n} x |\check{g} (t, x) |^2 dx =  \int_{\mathbb{R}^n} x |\check{\psi} (x) |^2 dx + t ( \nu M (\check{\psi}) - 2 \lambda P(\check{\psi} ) ) 
\end{equation*}
where $M (\check{\psi}) = \int_{\mathbb{R}^n}  |\check{\psi}|^2 dx  $, $P (\check{\psi}) =  i \int_{\mathbb{R}^n} \nabla \check{\psi} \cdot \overline{ \check{\psi}}  dx  $. Notice that $\int_{\mathbb{R}^n} x |\check{g}|^2 dx $ is a conserved quantity, we must have 
$$ \nu = \frac{2 \lambda P(\check{\psi} ) }{M (\check{\psi})} . $$
By invariance of $\mathcal{T}$ under translations, we can denote
\begin{equation}
\phi (K) := \psi (K- \frac{\nu}{2 \lambda} )
\end{equation}
so $\phi (K) $ solves
\begin{equation}  \label{StationaryWave-eq1}
(\mu + \lambda |K|^2 ) \phi = \mathcal{T} (\phi, \phi, \phi ) .
\end{equation}

\begin{lemma} [Energy and Pohozaev identities] \label{L:BGHS2-SS4.1-lm1}
Assume that $\phi$ solves \eqref{StationaryWave-NLS-eq1}, and that $\mathcal{\phi}$ is finite. If furthermore $\phi \in L^{2, 1}$, then it satisfies the energy identity
\begin{equation}   \label{StationaryWave-eq2}
\lambda \| K \phi \|_{L^2}^2 + \mu \|  \phi \|_{L^2}^2 = \mathcal{H} (\phi) . 
\end{equation}
If furthermore $\phi \in L^2$, $\xi \nabla \phi \in L^2$, then it satisfies the Pohozaev identity
\begin{equation}   \label{StationaryWave-eq3}
\lambda (\frac{n}{2} -1) \| K \phi \|_{L^2}^2 + \mu \frac{n}{2} \|  \phi \|_{L^2}^2 = (\frac{1}{2} + \frac{n}{4}) \mathcal{H} (\phi) . 
\end{equation}
\end{lemma}

\begin{proof}
The proof is similar to the one to Subsection 4.1, Lemma 1 in \cite{BGHS2} so we omit it. 
\end{proof}

Simple algebraic combinations of \eqref{StationaryWave-eq2} and \eqref{StationaryWave-eq3} yields:  

\begin{corollary}  \label{Cor:BGHS2-SS4.1-lm1-cor}
Assume that $\phi$ satisfies all the conditions in Lemma \ref{L:BGHS2-SS4.1-lm1}, then we have
\begin{equation}  \label{StationaryWave-eq4}
\lambda \| K \phi \|_{L^2}^2 = \frac{n-2}{4} \mathcal{H} (\phi) , \  \mu \|  \phi \|_{L^2}^2 = \frac{6-n}{4} \mathcal{H} (\phi) . 
\end{equation}
In particular, necessary conditions for \eqref{StationaryWave-eq1} to admit a solution are  \\
1) If $n=2$: $\lambda =0$ and $\mu>0$;  \\
2) If $3 \leq n \leq 5$: $\lambda >0$ and $\mu>0$;  \\
3) If $n=6$: $\lambda >0$ and $\mu=0$;  \\
4) If $n=7$: $\lambda >0$ and $\mu<0$.  \\
\end{corollary}



Let us consider the variational problems  \\
1) $\displaystyle \sup_{\|g \|_{L^2}^2 =1} \mathcal{H} (g)$ if $n=2$;  \\
2) $\displaystyle \sup_{\|g \|_{L^2}^2 +\| K g \|_{L^2}^2 =1} \mathcal{H} (g)$ if $3 \leq n \leq 5$;  \\
3) $\displaystyle \sup_{\| K g \|_{L^2}^2 =1} \mathcal{H} (g)$ if $n=6$.  \\
These problems make sense due to the Strichartz estimates
\begin{equation}
\begin{split}
& \mathcal{H} (g) \lesssim  2 \pi \| e^{ i t \Delta } \check{g} \|_{L^4_{t, x}}^4 \lesssim \|g \|_{L^2}^4  \text{ if } n=2 ,  \\
& \mathcal{H} (g) \lesssim (2 \pi)^{n-1} \| e^{ i t \Delta } \check{g} \|_{L^4_{t, x}}^4 \lesssim \|g \|_{L^{2, 1}}^4  \text{ if } 3 \leq n \leq 5 ,  \\
& \mathcal{H} (g) \lesssim (2 \pi)^{5} \| e^{ i t \Delta } \check{g} \|_{L^4_{t, x}}^4 \lesssim \|g \|_{\dot{L}^{2,1}}^4  \text{ if } n=6 .  \\
\end{split}
\end{equation}
Therefore, these variational problems belong to the class which arises from Fourier restriction functionals. 

The Euler-Lagrange equations satisfied by the maximizers of these variational problems read   
\begin{equation}
\begin{split}
& \lambda g = \mathcal{T} (g, g, g) \text{ if } n=2 ,  \\
& \lambda [ g + |K|^2 g ]  = \mathcal{T} (g, g, g) \text{ if } 3 \leq n \leq 5 , \\
& \lambda |K|^2 g = \mathcal{T} (g, g, g) \text{ if } n=6 ,  \\
\end{split}
\end{equation}
where $\lambda$ is the Lagrange multiplier. These three equations should be understood as equations in $L^2$, $H^{-1}$ and $\dot{H}^{-1}$, respectively. Since the maximizers are nonzero, testing the above equations against $g$ gives $\lambda > 0$. We can take $\lambda=1$ by scaling, and the Euler-Lagrange equations become   
\begin{equation}  \label{StationaryWave-Varia-eq1}
\begin{split}
&  g = \mathcal{T} (g, g, g) \text{ if } n=2 ,  \\
&  g + |K|^2 g   = \mathcal{T} (g, g, g) \text{ if } 3 \leq n \leq 5 , \\
&  |K|^2 g = \mathcal{T} (g, g, g) \text{ if } n=6 .  \\
\end{split}
\end{equation}

\begin{theorem}  \label{StationaryWave-Varia-th1}
The following variational problems admit nonzero maximizers:  \\
1) $\displaystyle \sup_{\|g \|_{L^2}^2 =1} \mathcal{H} (g)$ if $n=2$;  \\
2) $\displaystyle \sup_{\|g \|_{L^2}^2 +\| K g \|_{L^2}^2 =1} \mathcal{H} (g)$ if $3 \leq n \leq 5$;  \\
3) $\displaystyle \sup_{\|K g \|_{L^2}^2 =1} \mathcal{H} (g)$ if $n=6$.  \\
Furthermore, maximizing sequences are compact modulo the symmetries of the equation. For the case $n=2$, maximizers are given by tensor products of Gaussians.
\end{theorem} 

\begin{proof}


The case $n=2$ can be dealt with in a component-by-component way by refering to, for example, \cite{FGH1} and \cite{Shao1}. For the case $3 \leq n \leq 5$, we argue as in Theorem 1 Section 4.2 in \cite{BGHS2}. We take the Fourier transform of the above and consider the variational problem
\begin{equation}  \label{StationaryWave-Varia-th1-eq1}
\sup_{\|f \|_{L^2}^2 + \|\nabla f \|_{L^2}^2 = 1} \| e^{it \Delta} f \|_{L^4_{t, x}}^4 .
\end{equation}
For any $A > 0$, let $I(A)= \sup_{\|f \|_{L^2}^2 + \|\nabla f \|_{L^2}^2 = A} \| e^{it \Delta} f \|_{L^4_{t, x}}^4 $, then 
\begin{equation}   \label{StationaryWave-Varia-th1-eq2}
I(A) = A^2 I(1) .
\end{equation}
We make use of a profile expansion to exploit the compactness and find a maximizer. Consider $(f^q)_{q=1}^\infty$ ($f^q = (f^q_1, \cdots , f^q_n )$) a bounded sequence in $(H^1)^n$. There exists a subsequence (also denoted by $(f^q)_{q=1}^\infty$), and a second sequence $(\psi^k)_{k=1}^\infty$ ($\psi^k = (\psi^k_1, \cdots , \psi^k_n)$), and doubly indexed subsequences $(t_q^k)_{k, q =1}^\infty$, $(x_q^k)_{k, q =1}^\infty$ giving for any $k_0$ the decomposition
\begin{equation}
f^q = \sum_{k=1}^{k_0} e^{i t_q^k \Delta} \psi^k (x+ x_q^k) + r_q^{k_0} ,
\end{equation}
such that  \\
1) The expansion is orthogonal in the Strichartz norm:
\begin{equation}    \label{StationaryWave-Varia-th1-eq3}
\lim_{k_0 \rightarrow \infty} \limsup_{q \rightarrow \infty} \big( \| e^{it \Delta} f^q \|_{L^4_{t, x}}^4 -  \sum_{k=1}^{k_0} \| e^{it \Delta} \psi^k \|_{L^4_{t, x}}^4    \big)  = 0 . 
\end{equation}
2) The expansion is orthogonal in $L^2$: 
\begin{equation}    \label{StationaryWave-Varia-th1-eq4}
\text{For any } k_0 , \    \lim_{q \rightarrow \infty} \big( \|  f^q \|_{L^2_{x}}^2 -  \sum_{k=1}^{k_0} \|  \psi^k \|_{L^2_{x}}^2 -  \| r_q^{k_0} \|_{L^2_{x}}^2    \big)  = 0 . 
\end{equation}
3) The expansion is orthogonal in $\dot{H}^1$: 
\begin{equation}   \label{StationaryWave-Varia-th1-eq5}
\text{For any } k_0 , \    \lim_{q \rightarrow \infty} \big( \| \nabla  f^q \|_{L^2_{x}}^2 -  \sum_{k=1}^{k_0} \|  \nabla  \psi^k \|_{L^2_{x}}^2 -  \|  \nabla  r_q^{k_0} \|_{L^2_{x}}^2    \big)  = 0 . 
\end{equation}
The construction of the profile expansion can be carried out as follows: first consider $( f_1^q )_{q=1}^\infty$, performing profile expansion for this sequence gives a subsequence, also denoted by $( f_1^q )_{q=1}^\infty$ by re-indexing, as well as a second sequence $(\psi_1^k)_{k=1}^\infty$ and doubly indexed subsequences $(t_q^k)_{k, q =1}^\infty$, $(x_q^k)_{k, q =1}^\infty$. Then we take those elements in $( f_2^q )_{q=1}^\infty$ whose indices appear in the subsequence $( f_1^q )s_{q=1}^\infty$, and perform profile expansion, yielding a subsequence, also denoted by $( f_2^q )_{q=1}^\infty$ by re-indexing, as well as a second sequence $(\psi_2^k)_{k=1}^\infty$ and doubly indexed subsequences $(t_q^k)_{k, q =1}^\infty$, $(x_q^k)_{k, q =1}^\infty$ which are re-indexed subsequences of the $(t_q^k)_{k, q =1}^\infty$, $(x_q^k)_{k, q =1}^\infty$ determined in the previous step. Repeating this process $d$ times gives the profile decomposition for $(f^q)_{q=1}^\infty$, in which $\psi^k = (\psi^k_1, \cdots , \psi^k_n)$, and the indices come from the last step (on the $n$-th component).


Now, we pick $(f^q)_{q=1}^\infty$ a maximizing sequence for the variational problem \eqref{StationaryWave-Varia-th1-eq1} and perform the profile decomposition as described above. Then, due to the orthogonality property  \eqref{StationaryWave-Varia-th1-eq3} in the Strichartz norm,
\begin{equation}
I(1) = \lim_{q \rightarrow \infty} \| e^{it \Delta} f^q \|_{L^4_{t, x}}^4  =   \sum_{k=1}^{\infty} \| e^{it \Delta} \psi^k \|_{L^4_{t, x}}^4  .
\end{equation}
By the scaling property \eqref{StationaryWave-Varia-th1-eq2} of the variational problem, we have
\begin{equation}
I(1) \leq \sum_{k=1}^\infty I ( \| \psi^k \|_{L^2}^2 + \| \nabla \psi^k \|_{L^2}^2  ) \leq I(1)  \sum_{k=1}^\infty  ( \| \psi^k \|_{L^2}^2 + \| \nabla \psi^k \|_{L^2}^2  )^2 . 
\end{equation}
This in turn implies 
\begin{equation}
1 \leq  \sum_{k=1}^\infty  ( \| \psi^k \|_{L^2}^2 + \| \nabla \psi^k \|_{L^2}^2  )^2 
\end{equation}
while
\begin{equation}
\sum_{k=1}^\infty   \| \psi^k \|_{L^2}^2 + \| \nabla \psi^k \|_{L^2}^2  \leq  1  . 
\end{equation}
This is only possible if only one of the $\psi^k$'s is nonzero, say, $\psi^1$, without loss of generality. In other words, the maximizing sequence is compact (modulo symmetries), and $\psi^1$ is the desired maximizer.

\end{proof}

Moreover, by the same procedure as in \cite{BGHS2} we have the following decay properties for the solutions of the Euler-Lagrange equations:

\begin{proposition}
1) If $n = 2$, a solution $ g \in L^2$ of $g = \mathcal{T} (g,g,g) $ (where $\mathcal{T}$ is as given in \eqref{def-T-VNLS})
belongs to $L^{2,s}$ for all $s > 0$. \\
2) If $3 \leq n \leq 5$, a solution $ g \in L^{2,1}$ of $g+|K|^2 g=\mathcal{T}(g,g,g)$ (where $\mathcal{T}$ is as given in \eqref{def-T-VNLS}) belongs to $L^{2,s}$ for all $s>0$.
\end{proposition}

\section{Analysis for \eqref{CR-NLS}}
\label{S:CR-NLS-analysis}

This section is devoted to some analysis on the equation \eqref{CR-NLS}. Most of the results are parallel to the ones in Section \ref{S:CR-analysis} unless otherwise stated, so we omit these proofs.

\subsection{Hamiltonian Structure for \eqref{CR-NLS}}
\label{S:HamilStructure-NLS}

Let $n \geq 3$, $\mathcal{T} = (\mathcal{T}_j )_{j=1, 2, \cdots , d} $, 
\begin{equation*}
\mathcal{T}_j (f, g, h) (K) =  \int_{\mathbb{R}^{3n}} \big(\sum_{l=1}^d f_l (K_1) \overline{g}_l (K_2) \big) h_j (K_3) \delta_{\mathbb{R}^n} (S_3 (K)) \delta_{\mathbb{R}} (\Omega_3 (K)) dK_1 dK_2 dK_3   . 
\end{equation*} 
Define
\begin{equation} \label{D:E-NLS}
\begin{split}
\mathcal{E} (f, g, h, p)  
& := \int_{\mathbb{R}^{4n}}    \big(\sum_{l=1}^d  f_l (K_1) \overline{g}_l (K_2) \big) \big( \sum_{m=1}^d   h_m (K_3) \overline{p}_m (K) \big)  \\
& \quad \cdot  \delta_{\mathbb{R}^n} (S_3 (K)) \delta_{\mathbb{R}} (\Omega_3 (K)) dK_1 dK_2 dK_3  dK . \\
\end{split}
\end{equation} 
then the (CR) equation \eqref{CR-NLS} can be derived from the Hamiltonian
\begin{equation}
\mathcal{E} (g) := \mathcal{E} (g, g, g,g) ,
\end{equation}
given the symplectic form $\omega (f, g) = - 4 \Im \la f, g \ra_{L^2 (\mathbb{R}^n)} = 4  \int (\Re f \Im g - \Im f \Re g) =  \la f, S g  \ra_{L^2 (\mathbb{R}^n)} $, with $S g : = 4 ( \Im g + i \Re g )$. In other words, \eqref{CR-NLS} can also be written as 
\begin{equation}
-i \partial_t g = \frac{1}{2} \nabla_{\overline{g}} \mathcal{E} (g) . 
\end{equation}


From now on, we use $\check{f}$ to denote the inverse Fourier transform of $f$, and denote the harmonic oscillator $H= - \Delta + |x|^2$, and $\Pi_n$ be the projector on the $n$-th eigenspace of $H$.

We have
\begin{lemma} \label{L:lm2.2-GHT1-NLS}
The quantity $\mathcal{E}$ satisfies
\begin{equation} \label{E:lm2.2-GHT1-NLS-eq1}
\mathcal{E} (f, g, h, p) = (2\pi)^{n-1}  \int_{\mathbb{R}} \int_{\mathbb{R}^n}  \big[ \sum_{l=1}^d e^{it \Delta} \check{f}_l (x) \overline{e^{it \Delta} \check{g}_l (x)} \big] \big[ \sum_{m =1}^d e^{it \Delta} \check{h}_m (x) \overline{e^{it \Delta}    \check{p}_m (x)} \big] dx dt . 
\end{equation}
\begin{equation} \label{E:lm2.2-GHT1-NLS-eq2}
\mathcal{E} (f, g, h, p) = (2 \pi)^{n-1}   \int^{\pi/4}_{-\pi/4} \int_{\mathbb{R}^n}  \big[ \sum_{l=1}^d e^{-itH} \check{f}_l \overline{e^{-itH} \check{g}_l} \big] \big[ \sum_{m =1}^d e^{-itH} \check{h}_m \overline{ e^{-itH} (  \check{p}_m )} \big] (t, x) dx dt  . 
\end{equation}
Consequently,  
\begin{equation} \label{E:lm2.2-GHT1-NLS-eq3}
\mathcal{T}_j (f, g, h) = (2\pi)^{n-1} \mathcal{F} \Big\{   \int_{\mathbb{R}}        e^{-it \Delta}   \big[   \big( \sum_{l=1}^d  e^{it \Delta} \check{f}_l (x) \overline{e^{it \Delta} \check{g}_l (x)} \big) e^{it \Delta} \check{h}_j (x) \big]    dt \Big\} . 
\end{equation}
\begin{equation} \label{E:lm2.2-GHT1-NLS-eq4}
\mathcal{T}_j (f, g, h) = (2 \pi)^{n-1} \mathcal{F}  \Big\{   \int^{\pi/4}_{-\pi/4}     e^{it H} \big[  \big( \sum_{l=1}^d e^{-itH} \check{f}_l \overline{e^{-itH} \check{g}_l} \big) e^{-itH} \check{h}_j  (t, x) \big]   dt  \Big\} . 
\end{equation}
\end{lemma}

\begin{lemma} \label{L:lm2.3-GHT1-NLS}
The quantity $\mathcal{E}$ satisfies
\begin{equation} \label{E:lm2.3-GHT1-NLS-eq1}
\begin{split}
& \mathcal{E} (f, g, h, p) =    \frac{(2\pi)^{n}}{4}  \sum_{\substack{\ell_1 , \ell_2 , \ell_3 , \ell_4 \in \mathbb{Z}_{\geq 0} \\ \ell_1 + \ell_3 = \ell_2 + \ell_4}}  \int_{\mathbb{R}^n}  \big[ \sum_{l=1}^d  \Pi_{\ell_1} \check{f}_l (x)  \overline{\Pi_{\ell_2}  \check{g}_l (x) } \big] \big[  \sum_{m =1}^d   \Pi_{\ell_3}  \check{h}_m (x) \overline{\Pi_{\ell_4}  \check{p}_m (x)} \big] dx  . 
\end{split}
\end{equation}
Hence
\begin{equation} \label{E:lm2.3-GHT1-NLS-eq2}
\begin{split}
& \mathcal{T}_j (f, g, h) =   \frac{(2\pi)^{n}}{4}  \sum_{\substack{\ell_1 , \ell_2 , \ell_3 , \ell_4 \in \mathbb{Z}_{\geq 0} \\ \ell_1 + \ell_3 = \ell_2 + \ell_4}}   \mathcal{F} \Big\{  \Pi_{\ell_4}   \Big[  \big[ \sum_{l=1}^d  \Pi_{\ell_1} \check{f}_l (x)  \overline{ \Pi_{\ell_2}  \check{g}_l (x) } \big] \Pi_{\ell_3}  \check{h}_j (x) \Big] \Big\}     . 
\end{split}
\end{equation}
\end{lemma}

\subsection{Symmetries of $\mathcal{T}$ and Conservation Laws for \eqref{CR-NLS}}
\label{S:SymConservation-NLS} 
 
We observe
 
\begin{lemma}   \label{L:symmetry-NLS}
The following symmetries leave the Hamiltonian $\mathcal{E}$ invariant: \\
1) Rotation: $g \mapsto U g$ for all any $U$ in the unitary group $U(d)$. In particular, phase rotation: $g = (g_1, g_2, \cdots, g_d) \mapsto ( e^{i \theta_1} g_1 , e^{i \theta_2} g_2 , \cdots , e^{i \theta_d} g_d ) $ for all $\theta_1, \cdots, \theta_d \in \mathbb{R}$; \\
2) Translation: $g \mapsto  g (\cdot + K_0) $ for all $K_0 \in \mathbb{R}^n$; \\
3) Modulation: $g \mapsto e^{i K \cdot x_0} g$ for all $x_0 \in \mathbb{R}^n$; \\
4) Quadratic modulation: $g \mapsto e^{i \tau |K|^2} g$ for all $\tau \in \mathbb{R}$; \\
5) Rotation: $g \mapsto g(O \cdot ) $ for all any $O$ in the orthogonal group $O(n)$; \\
6) Scaling: $g \mapsto \lambda^{\frac{3n-2}{4}} g (\lambda \cdot ) $ for all $\lambda \in \mathbb{R} \setminus \{ 0 \}$. \\
\end{lemma}

\begin{proof}
By direct inspection.
\end{proof}

Hence by Noether’s theorem, we have the following conserved quantities associated to the symmetries 1) -- 5), respectively:

%
\begin{lemma} \label{L:conservation-law-NLS}
The following quantities are conserved by the flow of (CR):  \\
1) Rotation: 
\begin{equation*}
\begin{split}
& M_j = \int |g_j|^2 dK , \  j = 1, 2, \cdots , d ,  \\ 
& \Im \int g_j (K)  \overline{g_k (K)}  dK , \  j \neq k , \  j, k = 1, 2, \cdots , d ,  \\
& \Re \int g_j (K)  \overline{g_k (K)}  dK , \  j \neq k , \  j, k = 1, 2, \cdots , d .  \\
\end{split}
\end{equation*}
2) Translation: 
\begin{equation*}
\begin{split}
& \int x_j |\check{g} (x)|^2 dx  ,  \  j = 1, 2, \cdots , n .   \\ 
\end{split}
\end{equation*}
3) Modulation: 
\begin{equation*}
\begin{split}
& \int K (j) |g|^2 dK , \  j = 1, 2, \cdots , n .  \\
\end{split}
\end{equation*}
4) Quadratic modulation: 
\begin{equation*}
\begin{split}
& \int |K|^2 |g|^2 dK . \\
\end{split}
\end{equation*}
5) Rotation: 
\begin{equation*}
\begin{split}
& \Im  \int ( K(j) \partial_{K(k)} - K(k) \partial_{K(j)} ) g (K) \cdot \overline{g (K)} dK , \ j \neq k ,  \ j , k = 1, \cdots , n .  \\
\end{split}
\end{equation*}
\end{lemma}

\begin{proof}
The conservation laws can be obtained by observing the Lie groups associated with the symmetries in Lemma \ref{L:symmetry-NLS}. 

In fact, for 1) in Lemma \ref{L:symmetry-NLS}, the associated conserved quantities are given by
\begin{equation*}
\frac{1}{2}  \Im \int A g(K) \cdot \overline{g(K)} dK = \frac{1}{2}  \Im \sum_{j=1}^d \int (A g(K))_j  (\overline{g(K)})_j dK  ,
\end{equation*}
where $A$ is any $d \times d$ skew-Hermitian matrix (i.e. $A^* = -A$). Denote $e_{jk}^l$ as the $l \times l$ matrix with only the $j \times k$ element being $1$ and all the other elements being zero, then taking $A= i e_{jj}^d$ gives the conserved quantity $M_j = \int |g_j|^2 dK$, and taking $A= i ( e_{jk}^d + e_{kj}^d)$ gives the conserved quantity $\Re \int g_j (K)  \overline{g_k (K)}  dK$, taking $A=  e_{jk}^d - e_{kj}^d$ gives the conserved quantity $\Im \int g_j (K)  \overline{g_k (K)}  dK$. 

For 5) in Lemma \ref{L:symmetry-NLS}, the associated conserved quantities are given by
\begin{equation*}
\frac{1}{2}  \Im \sum_{j=1}^d \int ( \nabla_K g_j (K) \cdot A K ) \overline{g_j (K)}  dK  ,
\end{equation*}
where $A$ is any $n \times n$ skew-orthogonal matrix (i.e. $A^\top = -A$). Taking $A=  e_{jk}^n - e_{kj}^n$ gives the conserved quantity $\Im  \int ( K(j) \partial_{K(k)} - K(k) \partial_{K(j)} ) g (K) \cdot \overline{g (K)} dK$. 

\end{proof}

\subsection{Well-Posedness for \eqref{CR-NLS}}
\label{S:LWP-NLS}

Recall that we denote
\begin{equation}
\begin{split}
& \|f \|_{L^{p, l}  } :=  \| \la K \ra^l f (K) \|_{L^p  }  , \  L^{p, l}_n := (L^{p, l})^n  ,  \\
& \|f \|_{\dot{L}^{p, l}  } :=  \| | K |^l f (K) \|_{L^p  }  , \  \dot{L}^{p, l}_n := (\dot{L}^{p, l})^n  ,  \\
& \|f \|_{W^{p, l}  } :=  \| \la D \ra^l f (K) \|_{L^p  }  , \  W^{p, l}_n := (W^{p, l})^n , \\
& \|f \|_{X^{l, N}  } := \sum_{0 \leq |\alpha| \leq N} \| \nabla^\alpha  f \|_{X^l  }  = \sum_{0 \leq |\alpha| \leq N} \| \nabla^\alpha  f \|_{L^{\infty, l} }    ,   \   X^{l, N}_n  := (X^{l, N}  )^n  .  \\
\end{split}
\end{equation}

We have the following proposition for the boundedness property of the operator $\mathcal{T} = (\mathcal{T}_j )_{j=1, 2, \cdots , n} $,
\begin{equation}  \label{def-T-NLS}
\mathcal{T}_j (g, g, g) (K) :=   \int_{\mathbb{R}^{3n}} \big(\sum_{l=1}^d g_l (K_1) \overline{g}_l (K_2) \big) g_j (K_3) \delta_{\mathbb{R}^n} (S_3 (K)) \delta_{\mathbb{R}} (\Omega_3 (K)) dK_1 dK_2 dK_3  .
\end{equation} 

\begin{proposition} \label{L:BGHS2-prop5-NLS}
The trilinear operator $\mathcal{T}$ in \eqref{def-T-NLS} is bounded from $X \times X \times X$ to $X$ for the following Banach spaces $X$:  \\
1) $X= \dot{L}^{p, \frac{n-2}{2}}_n $;  \\
2) $X= L^{p, l}_n $, $l \geq \frac{n-2}{2}$;  \\
3) $X= L^{\infty, l}_n $, $l  >  n-1$;  \\
4) $X= L^{p, l}_n $, $p \geq 2$, $l  >  n-1- \frac{n}{p}$;  \\
5) $X= X^{l, N}_n $, $l  >  n-1$, $N \geq 0$.  \\
\end{proposition}

\begin{remark}
The borderline spaces for well-posedness above are $ \dot{L}^{p, n-1-\frac{n}{p}}_n$. They share the same scaling, and are also scale-invariant for the cubic NLS in $n$-dimension \eqref{NLS} (when viewed as spaces for $\hat{u}$).
\end{remark}

\begin{proof}
The proof is very similar to the one for Proposition 5 in \cite{BGHS2} so we omit it. We can perform an argument which is essentially the same as in Proposition 5 in \cite{BGHS2} for each $\mathcal{T}_j$. 
\end{proof}

From Proposition \ref{L:BGHS2-prop5-NLS}, we conclude the following local well-posedness results:

\begin{theorem}  \label{L:BGHS2-prop5-cor1-1-NLS}
1) For $n \geq 3$, $X$ any spaces given in Proposition \ref{L:BGHS2-prop5-NLS}, the equation \eqref{CR-NLS} is locally well-posed in $X$;  \\
2) For $n=2$, $X$ any spaces given in Proposition \ref{L:BGHS2-prop5-NLS}, the equation $-i\partial_t g(t,\xi)= \mathcal{T} (g(t, \cdot) , g(t, \cdot) , g(t, \cdot)) (t, \xi) $ with $\mathcal{T}$ in \eqref{def-T-NLS} is locally well-posed in $X$.  \\
\end{theorem}

\begin{remark}
For the well-posedness of equation \eqref{CR-NLS-n=2} in the case $n=2$, we also need to consider the last term on the right hand side in \eqref{CR-NLS-n=2}, which, for the time being, lack of understanding and sufficient estimate for local well-posedness.
\end{remark}

Combining the local well-posedness and the conservation laws together, we obtain the following results:

\begin{corollary}  \label{L:BGHS2-prop5-cor1-2-NLS}
With the assumptions in Theorem \ref{L:BGHS2-prop5-cor1-1-NLS}, we have  \\
1) If $n=2$ and $g_0 \in (L^2)^n$, then the local solution can be prolonged into a global one;  \\
2) If $n=2, 3, 4$ and $g_0 \in L^{2, 1}_n$, then the local solution can be prolonged into a global one.  \\
3) If $g$ is a solution to the equation \eqref{CR-NLS}, $g_0 \in \dot{L}^{p, \frac{n-2}{2}}_n $, and $g_0 \in L^{p, l}_n $ for any $l \geq \frac{n-2}{2}$, then $g \in C^\infty ([0, T), L^{2, l}_n)$. 
\end{corollary}

\subsection{Stationary Waves of \eqref{CR-NLS}}
\label{S:StationaryWave-NLS}

In this section, we consider the existence of solutions of the CR equation \eqref{CR-NLS} of the $n$-dimensional NLS of the type
\begin{equation}
g(t, K) = e^{-i(\mu + \lambda |K|^2 + \nu \cdot K)t } \psi (K) ,
\end{equation}
where $\lambda$, $\mu \in \mathbb{R}$, $\nu \in \mathbb{R}^n$. For $g$ to solve \eqref{CR-NLS}, it suffices that $\psi$ solves
\begin{equation}
(\mu + \lambda |K|^2 + \nu \cdot K) \psi = \mathcal{T} (\psi, \psi, \psi ) 
\end{equation}
where $\mathcal{T}$ is as given in \eqref{def-T-NLS}.
Notice that $g$ defined above oscillates in Fourier space, but it actually travels in physical space, as can be seen by taking its inverse Fourier transform:
\begin{equation}
\check{g} (t, x) = e^{-i(\mu - \lambda \Delta )t } \check{\psi} (x- \nu t) ,
\end{equation}
The conservation of position gives a restriction on the relation between $\nu$ and $\lambda$. Using the identity $[x, e^{it \Delta} ] = -2 it \nabla e^{it \Delta} $, we have
\begin{equation*}
\int_{\mathbb{R}^n} x |\check{g} (t, x) |^2 dx =  \int_{\mathbb{R}^n} x |\check{\psi} (x) |^2 dx + t ( \nu M (\check{\psi}) - 2 \lambda P(\check{\psi} ) ) 
\end{equation*}
where $M (\check{\psi}) = \int_{\mathbb{R}^n}  |\check{\psi}|^2 dx  $, $P (\check{\psi}) =  i \int_{\mathbb{R}^n} \nabla \check{\psi} \cdot \overline{ \check{\psi}}  dx  $. Notice that $\int_{\mathbb{R}^n} x |\check{g}|^2 dx $ is a conserved quantity, we must have 
$$ \nu = \frac{2 \lambda P(\check{\psi} ) }{M (\check{\psi})} . $$
By invariance of $\mathcal{T}$ under translations, we can denote
\begin{equation}
\phi (K) := \psi (K- \frac{\nu}{2 \lambda} )
\end{equation}
so $\phi (K) $ solves
\begin{equation}  \label{StationaryWave-NLS-eq1}
(\mu + \lambda |K|^2 ) \phi = \mathcal{T} (\phi, \phi, \phi ) .
\end{equation}

\begin{lemma} [Energy and Pohozaev identities] \label{L:BGHS2-SS4.1-lm1-NLS}
Assume that $\phi$ solves \eqref{StationaryWave-NLS-eq1}, and that $\mathcal{\phi}$ is finite. If furthermore $\phi \in L^{2, 1}$, then it satisfies the energy identity
\begin{equation}   \label{StationaryWave-NLS-eq2}
\lambda \| K \phi \|_{L^2}^2 + \mu \|  \phi \|_{L^2}^2 = \mathcal{H} (\phi) . 
\end{equation}
If furthermore $\phi \in L^2$, $\xi \nabla \phi \in L^2$, then it satisfies the Pohozaev identity
\begin{equation}   \label{StationaryWave-NLS-eq3}
\lambda (\frac{n}{2} -1) \| K \phi \|_{L^2}^2 + \mu \frac{n}{2} \|  \phi \|_{L^2}^2 = (\frac{1}{2} + \frac{n}{4}) \mathcal{H} (\phi) . 
\end{equation}
\end{lemma}

\begin{proof}
The proof is similar to the one to Subsection 4.1, Lemma 1 in \cite{BGHS2}.
\end{proof}

Simple algebraic combinations of \eqref{StationaryWave-NLS-eq2} and \eqref{StationaryWave-NLS-eq3} yields:  

\begin{corollary}  \label{Cor:BGHS2-SS4.1-lm1-NLS-cor}
Assume that $\phi$ satisfies all the conditions in Lemma \ref{L:BGHS2-SS4.1-lm1-NLS}, then we have
\begin{equation}  \label{StationaryWave-NLS-eq4}
\lambda \| K \phi \|_{L^2}^2 = \frac{n-2}{4} \mathcal{H} (\phi) , \  \mu \|  \phi \|_{L^2}^2 = \frac{6-n}{4} \mathcal{H} (\phi) . 
\end{equation}
In particular, necessary conditions for \eqref{StationaryWave-NLS-eq1} to admit a solution are  \\
1) If $n=2$: $\lambda =0$ and $\mu>0$;  \\
2) If $3 \leq n \leq 5$: $\lambda >0$ and $\mu>0$;  \\
3) If $n=6$: $\lambda >0$ and $\mu=0$;  \\
4) If $n=7$: $\lambda >0$ and $\mu<0$.  \\
\end{corollary}



Let us consider the variational problems  \\
1) $\displaystyle \sup_{\|g \|_{L^2}^2 =1} \mathcal{H} (g)$ if $n=2$;  \\
2) $\displaystyle \sup_{\|g \|_{L^2}^2 +\| K g \|_{L^2}^2 =1} \mathcal{H} (g)$ if $3 \leq n \leq 5$;  \\
3) $\displaystyle \sup_{\| K g \|_{L^2}^2 =1} \mathcal{H} (g)$ if $n=6$.  \\
These problems make sense due to the Strichartz estimates
\begin{equation}
\begin{split}
& \mathcal{H} (g) \lesssim 2 \pi \| e^{ i t \Delta } \check{g} \|_{L^4_{t, x}}^4 \lesssim \|g \|_{L^2}^4  \text{ if } n=2 ,  \\
& \mathcal{H} (g) \lesssim (2 \pi)^{n-1} \| e^{ i t \Delta } \check{g} \|_{L^4_{t, x}}^4 \lesssim \|g \|_{L^{2, 1}}^4  \text{ if } 3 \leq n \leq 5 ,  \\
& \mathcal{H} (g) \lesssim (2 \pi)^{5} \| e^{ i t \Delta } \check{g} \|_{L^4_{t, x}}^4 \lesssim \|g \|_{\dot{L}^{2,1}}^4  \text{ if } n=6 .  \\
\end{split}
\end{equation}
Therefore, these variational problems belong to the class which arises from Fourier restriction functionals. 

The Euler-Lagrange equations satisfied by the maximizers of these variational problems read   
\begin{equation}
\begin{split}
& \lambda g = \mathcal{T} (g, g, g) \text{ if } n=2 ,  \\
& \lambda [ g + |K|^2 g ]  = \mathcal{T} (g, g, g) \text{ if } 3 \leq n \leq 5 , \\
& \lambda |K|^2 g = \mathcal{T} (g, g, g) \text{ if } n=6 ,  \\
\end{split}
\end{equation}
where $\lambda$ is the Lagrange multiplier. These three equations should be understood as equations in $L^2$, $H^{-1}$ and $\dot{H}^{-1}$, respectively. Since the maximizers are nonzero, testing the above equations against $g$ gives $\lambda > 0$. We can take $\lambda=1$ by scaling, and the Euler-Lagrange equations become   
\begin{equation}  \label{StationaryWave-NLS-Varia-eq1}
\begin{split}
&  g = \mathcal{T} (g, g, g) \text{ if } n=2 ,  \\
&  g + |K|^2 g   = \mathcal{T} (g, g, g) \text{ if } 3 \leq n \leq 5 , \\
&  |K|^2 g = \mathcal{T} (g, g, g) \text{ if } n=6 .  \\
\end{split}
\end{equation}

\begin{theorem}  \label{StationaryWave-NLS-Varia-th1}
The following variational problems admit nonzero maximizers:  \\
1) $\displaystyle \sup_{\|g \|_{L^2}^2 =1} \mathcal{H} (g)$ if $n=2$;  \\
2) $\displaystyle \sup_{\|g \|_{L^2}^2 +\| K g \|_{L^2}^2 =1} \mathcal{H} (g)$ if $3 \leq n \leq 5$;  \\
3) $\displaystyle \sup_{\|K g \|_{L^2}^2 =1} \mathcal{H} (g)$ if $n=6$.  \\
Furthermore, maximizing sequences are compact modulo the symmetries of the equation. For the case $n=2$, maximizers are given by tensor products of Gaussians.
\end{theorem}

Moreover, by the same procedure as in \cite{BGHS2} we have the following decay properties for the solutions of the Euler-Lagrange equations:

\begin{proposition}
1) If $n = 2$, a solution $ g \in L^2$ of $g = \mathcal{T} (g,g,g) $ (where $\mathcal{T}$ is as given in \eqref{def-T-NLS})
belongs to $L^{2,s}$ for all $s > 0$. \\
2) If $3 \leq n \leq 5$, a solution $ g \in L^{2,1}$ of $g+|K|^2 g=\mathcal{T}(g,g,g)$ (where $\mathcal{T}$ is as given in \eqref{def-T-NLS}) belongs to $L^{2,s}$ for all $s>0$.
\end{proposition}

\subsection{Brief Dynamics of \eqref{CR-NLS} on the eigenspaces of $H$ when $n=2$}
\label{S:Dynamics-Eigen-NLS}

In this section, we consider the equation \eqref{CR-NLS} for $n= 2$. The special Hermite functions are defined as 
\begin{equation*}
\psi_{l, m} = \frac{1}{\sqrt{\pi l! m!}} (a_d^\star)^l (a_g^\star)^m e^{- z\overline{z}/2}
\end{equation*}
and for $l+m$ even we set
\begin{equation*}
\varphi_{l, m} = \psi_{\frac{l+m}{2}, \frac{l-m}{2}} . 
\end{equation*}
Here
\begin{equation*}
a_d^\star = \frac{1}{2} (z - 2 \partial_{\overline{z}}) , \  a_g^\star = \frac{1}{2} (\overline{z} - 2 \partial_z) , \  \partial_z = \frac{1}{2} (\partial_{x_1}- i \partial_{x_2}) , \  \partial_{\overline{z}} = \frac{1}{2} (\partial_{x_1}+ i \partial_{x_2}) . 
\end{equation*}
There holds

\begin{proposition}
Let $H= - \Delta + |x|^2$, $L= i ( x \times \nabla) $. \\
1) The family $(\psi_{l, m})_{l \geq 0,  m \geq 0}$ is an $L^2$-normalized Hilbertian basis of the space $L^2 (\mathbb{R}^2)$ such that
\begin{equation*}
H \psi_{l, m} = 2 (l+m +1) \psi_{n, m} , \  L \psi_{l, m} = (l-m) \psi_{l, m} , \   \hat{\psi}_{l, m} = e^{-i (l+m) \pi/2 } \psi_{l, m} . 
\end{equation*}
2) The family $(\varphi_{l, m})_{l \geq 0, -l \leq m \leq l , l+m \text{ even}}$ is an $L^2$-normalized Hilbertian basis of the space $L^2 (\mathbb{R}^2)$ such that
\begin{equation*}
H \varphi_{l, m} = 2 (l +1) \varphi_{l, m} , \  L \varphi_{l, m} = m \varphi_{l, m} , \   \hat{\varphi}_{l, m} = e^{-i l \pi/2 } \varphi_{l, m} . 
\end{equation*}
\end{proposition}
The definitions and properties above are presented in detail and explained in \cite{GHTho1}, Section 5 so we omit the details here. Define the eigenspace $E_l = \text{span} (\varphi_{l, m},  \ l \geq 0, \ -l \leq m \leq l , \ l+m \text{ even}) $, so $E_l$ is the $l$-th eigenspace of $H = -\Delta + |x|^2$, associated to the eigenvalue $2l+2$. Let $u \in L^2 ( \mathbb{R}^2)$ so $u$ can be written as
\begin{equation*}
u = \sum_{l=0}^\infty u_l , \ u_l = \sum_{m=-l}^l c_{l, m} \varphi_{l, m} \in E_l 
\end{equation*}
with the convention $c_{l, m} =0$ if $l+m$ is odd. For $u \in  (L^2 ( \mathbb{R}^2))^d$, we can write
\begin{equation*}
u_j = \sum_{l=0}^\infty u_{j, l} , \ u_{j, l} = \sum_{m=-l}^l c_{j, l, m} \varphi_{l, m} \in E_l , \ j = 1, 2, \cdots , d . 
\end{equation*}
We can describe the trilinear operator $\mathcal{T}$ in \eqref{def-T-NLS} and the Hamiltonian $\mathcal{E}$ in \eqref{D:E-NLS} with the basis of special Hermite functions:
\begin{proposition}
Let $e_j$ being the $d$-dimensional vector with only the $j$-th element equal to $1$ and all other elements being $0$. There holds:  
$$\mathcal{E} ( \varphi_{l_1, m_1} e_{j_1} , \varphi_{l_2, m_2} e_{j_2} , \varphi_{l_3, m_3} e_{j_3} , \varphi_{l_4, m_4} e_{j_4} ) = \pi^2 ( \int_{\mathbb{R}^2} \varphi_{l_1, m_1} \varphi_{l_2, m_2}   \overline{\varphi_{l_3, m_3}} \overline{\varphi_{l_4, m_4}} dx   )  \mathbf{1}_{\substack{ j_1 = j_2, j_3= j_4 , \\  l_1 + l_2 = l_3 + l_4 , \\ m_1 + m_2 = m_3 + m_4}} $$
and
$$\mathcal{T} ( \varphi_{l_1, m_1} e_{j_1} , \varphi_{l_2, m_2} e_{j_2} , \varphi_{l_3, m_3} e_{j_3}   ) =  \pi^2 ( \int_{\mathbb{R}^2} \varphi_{l_1, m_1} \varphi_{l_2, m_2} \overline{\varphi_{l_3, m_3}} \overline{\varphi_{l_4, m_4}} dx   ) \varphi_{l_4, m_4} e_{j_4} $$
with $l_4 = l_1 + l_2 - l_3$, $m_4 = m_1 + m_2 - m_3$, $j_4= j_3$. In fact, $\int_{\mathbb{R}^2} \varphi_{l_1, m_1}  \varphi_{l_2, m_2} \overline{ \varphi_{l_3, m_3} } \overline{ \varphi_{l_4, m_4} } dx =0 $ if $m_1 + m_2 \neq m_3 + m_4$. 
\end{proposition}

\begin{proof}

Let us first state the following lemma, whose proof is similar to the ones for Lemma 2.1, 2.2, 2.3 in \cite{GHTho1}. 

\begin{lemma} \label{L:lm2.2-2.3-GHT1-NLS}
The quantity $\mathcal{E}$ satisfies
\begin{equation} \label{E:lm2.2-2.3-GHT1-NLS-eq1}
\mathcal{E} (f, g, h, p) = (2\pi)^{n-1}  \int_{\mathbb{R}} \int_{\mathbb{R}^n}  \big[ \sum_{l=1}^d e^{it \Delta} f_l (x) \overline{e^{it \Delta} g_l (x)} \big] \big[ \sum_{m =1}^d e^{it \Delta} h_m (x) \overline{e^{it \Delta}    p_m (x)} \big] dx dt . 
\end{equation}
\begin{equation} \label{E:lm2.2-2.3-GHT1-NLS-eq2}
\mathcal{E} (f, g, h, p) = (2 \pi)^{n-1}   \int^{\pi/4}_{-\pi/4} \int_{\mathbb{R}^n}  \big[ \sum_{l=1}^d e^{-itH} f_l \overline{e^{-itH} g_l} \big] \big[ \sum_{m =1}^d e^{-itH} h_m \overline{ e^{-itH} ( p_m )} \big] (t, x) dx dt  ,
\end{equation}
\begin{equation} \label{E:lm2.2-2.3-GHT1-NLS-eq5}
\begin{split}
& \mathcal{E} (f, g, h, p) =    \frac{(2\pi)^{n}}{4}  \sum_{\substack{\ell_1 , \ell_2 , \ell_3 , \ell_4 \in \mathbb{Z}_{\geq 0} \\ \ell_1 + \ell_3 = \ell_2 + \ell_4}}  \int_{\mathbb{R}^n}  \big[ \sum_{l=1}^d  \Pi_{\ell_1} f_l (x)  \overline{\Pi_{\ell_2}  g_l (x) } \big] \big[  \sum_{m =1}^d   \Pi_{\ell_3}  h_m (x) \overline{\Pi_{\ell_4}  p_m (x)} \big] dx  . 
\end{split}
\end{equation}
Consequently,  
\begin{equation} \label{E:lm2.2-2.3-GHT1-NLS-eq3}
\mathcal{T}_j (f, g, h) = (2\pi)^{n-1} \mathcal{F} \Big\{   \int_{\mathbb{R}}        e^{-it \Delta}   \big[   \big( \sum_{l=1}^d  e^{it \Delta} f_l (x) \overline{e^{it \Delta} g_l (x)} \big) e^{it \Delta} h_j (x) \big]    dt \Big\} . 
\end{equation}
\begin{equation} \label{E:lm2.2-2.3-GHT1-NLS-eq4}
\mathcal{T}_j (f, g, h) = (2 \pi)^{n-1} \mathcal{F}  \Big\{   \int^{\pi/4}_{-\pi/4}     e^{it H} \big[  \big( \sum_{l=1}^d e^{-itH} f_l \overline{e^{-itH} g_l} \big) e^{-itH} h_j  (t, x) \big]   dt  \Big\} ,
\end{equation}
\begin{equation} \label{E:lm2.2-2.3-GHT1-NLS-eq6}
\begin{split}
& \mathcal{T}_j (f, g, h) =   \frac{(2\pi)^{n}}{4}  \sum_{\substack{\ell_1 , \ell_2 , \ell_3 , \ell_4 \in \mathbb{Z}_{\geq 0} \\ \ell_1 + \ell_3 = \ell_2 + \ell_4}}   \mathcal{F} \Big\{  \Pi_{\ell_4}   \Big[  \big[ \sum_{l=1}^d  \Pi_{\ell_1} f_l (x)  \overline{ \Pi_{\ell_2}  g_l (x) } \big] \Pi_{\ell_3} h_j (x) \Big] \Big\}    . 
\end{split}
\end{equation}
\end{lemma}

We claim that
\begin{equation*}
\mathcal{T} ( \varphi_{l_1, m_1} e_{j_1} , \varphi_{l_2, m_2} e_{j_2} , \varphi_{l_3, m_3} e_{j_3}   )  =  \mathcal{E} ( \varphi_{l_1, m_1} e_{j_1} , \varphi_{l_2, m_2} e_{j_2} , \varphi_{l_3, m_3} e_{j_3} , \varphi_{l_4, m_4} e_{j_4} ) \varphi_{l_4, m_4} e_{j_4} 
\end{equation*}
with $l_4 = l_1 + l_2 - l_3$, $m_4 = m_1 + m_2 - m_3$, $j_4= j_3$. Indeed, since $\varphi_{l, m}$ is an eigenfunction of $H$, $\mathcal{T} ( \varphi_{l_1, m_1} e_{j_1} , \varphi_{l_2, m_2} e_{j_2} , \varphi_{l_3, m_3} e_{j_3}   ) $ must be an eigenfunction of $H$ and be co-linear to $\varphi_{l_4, m_4} e_{j_4} $ due to Lemma \ref{L:lm2.2-2.3-GHT1-NLS}. The definition of $\mathcal{E}$ then gives the claim.

Moreover, by Lemma \ref{L:lm2.2-2.3-GHT1-NLS} we have
\begin{equation*}
\mathcal{E} ( \varphi_{l_1, m_1} e_{j_1} , \varphi_{l_2, m_2} e_{j_2} , \varphi_{l_3, m_3} e_{j_3} , \varphi_{l_4, m_4} e_{j_4} ) = 2 \pi \int_{-\pi/4}^{\pi/4} e^{-2i (l_1+l_2-l_3-l_4)t} dt \int_{\mathbb{R}^2} \varphi_{l_1, m_1} \varphi_{l_2, m_2}   \overline{\varphi_{l_3, m_3}} \overline{\varphi_{l_4, m_4}} dx
\end{equation*}
which yields
\begin{equation*}
\mathcal{E} ( \varphi_{l_1, m_1} e_{j_1} , \varphi_{l_2, m_2} e_{j_2} , \varphi_{l_3, m_3} e_{j_3} , \varphi_{l_4, m_4} e_{j_4} ) = \pi^2  \int_{\mathbb{R}^2} \varphi_{l_1, m_1} \varphi_{l_2, m_2}   \overline{\varphi_{l_3, m_3}} \overline{\varphi_{l_4, m_4}} dx
\end{equation*}
for $l_4 = l_1 + l_2 - l_3$, $m_4 = m_1 + m_2 - m_3$, $j_1= j_2$, $j_3= j_4$, and zero otherwise.
\end{proof}
Therefore, expanding $f = f(K) \in (L^2(\mathbb{R}^2))^d$ in the basis of special Hermite functions, the CR equation \eqref{CR-NLS} is equivalent to
\begin{equation}
i \dot{u}_l = \sum_{\substack{l_1, l_2, l_3 \geq 0 \\ l_1 + l_2 - l_3  = l} } \mathcal{T} (u_{l_1}, u_{l_2}, u_{l_3} ) ,
\end{equation}
or, for $l \geq 0$ and $-l \leq m \leq l$, 
\begin{equation}
i \dot{c}_{j, l, m} = \pi^2  \sum_{\substack{l_1, l_2, l_3 \geq 0 \\ l_1 + l_2 - l_3  = l} }  \sum_{\substack{ -l_k \leq m_k \leq l_k, k =1,2, 3 \\  m_1 + m_2 - m_3 = m \\ j_1 = j_2, j_3 = j } }   ( \int_{\mathbb{R}^2} \varphi_{l_1, m_1} \varphi_{l_2, m_2} \overline{\varphi_{l_3, m_3}} \overline{\varphi_{l, m}} dx   )  c_{j_1, l_1, m_1} c_{j_2, l_2, m_2} \overline{c_{j_3, l_3, m_3}}  .
\end{equation}


\subsubsection{Dynamics on $E_0$} 
The eigenspace $E_0$ is generated by the Gaussian $\varphi_{0, 0} (K) = \frac{1}{\sqrt{\pi}} e^{-\frac{1}{2} |K|^2}$. For the data $g_j (t=0)= c_{j, 0} \varphi_{0, 0}$, $j= 1, \cdots, d$, the solution $g_j (t)= c_{j} (t) \varphi_{0, 0}$ is given by $c_j (t) =  e^{- i \frac{\pi}{2} |c_{j, 0}|^2 t} c_{j, 0}$. 

\subsubsection{Dynamics on $E_1$} 
We write $g_j = c_{j, 1} \varphi_{1, 1} + c_{j, -1} \varphi_{1, -1} $, $j= 1, \cdots , d$. Then $\mathcal{E} (g) = \pi^2  \int_{\mathbb{R}^n} |g|^4$. Using Lemma \ref{L:lm2.3-GHT1-NLS}, we compute
\begin{equation}
\begin{split}
\mathcal{E} (u) 
& = \pi^2 \int_{\mathbb{R}^2} (\sum_{j=1}^d |c_{j, 1} \varphi_{1, 1} + c_{j, -1} \varphi_{1, -1}|^2)^2 \\
& = \pi^2  \int_{\mathbb{R}^2} \big\{  \sum_{j=1}^d |c_{j, 1} \varphi_{1, 1}|^4 +  \sum_{j=1}^d |c_{j, -1} \varphi_{1, -1}|^4  \\
& \quad + 2 \sum_{j, k=1}^d | c_{j, 1} \varphi_{1, 1} |^2 | c_{k, -1} \varphi_{1, -1} |^2  + 2 \sum_{j, k=1}^d c_{j, 1} \overline{c}_{k, 1}  | \varphi_{1, 1} |^2 c_{j, -1} \overline{c}_{k, -1}  | \varphi_{1, -1} |^2   \big\}  \\
& =   \frac{\pi}{4} \sum_{j=1}^d |c_{j, 1}  |^4 + \frac{\pi}{4} \sum_{j=1}^d |c_{j, -1}  |^4   +  \frac{\pi}{2} \sum_{j, k=1}^d | c_{j, 1}   |^2 | c_{k, -1}   |^2  + \frac{\pi}{2} \sum_{j, k=1}^d c_{j, 1} \overline{c}_{k, 1}  c_{j, -1} \overline{c}_{k, -1}  .  \\
\end{split}
\end{equation}
The CR equation $-i \partial_t g = \frac{1}{2} \nabla_{\overline{g}} \mathcal{E} (g) $ becomes ($j, k=1, \cdots , d$)
\[ \begin{cases} 
      i \dot{c}_{j, 1}= - \big\{ \frac{\pi}{4}  |c_{j, 1}|^2 c_{j, 1} + \frac{\pi}{2} |c_{j, -1}|^2 c_{j, 1}  + \frac{\pi}{4} \sum_{k \neq j}  | c_{k, -1}   |^2  c_{j, 1}     +  \frac{\pi}{4} \sum_{k \neq j}  c_{k, 1}  c_{k, -1} \overline{c}_{j, -1} \big\}\\
      i \dot{c}_{j, -1} = - \big\{  \frac{\pi}{4} |c_{j, -1}|^2  c_{j, -1} +  \frac{\pi}{2} |c_{j, 1}|^2 c_{j, -1}  + \frac{\pi}{4} \sum_{k \neq j} | c_{k, 1}   |^2   c_{j, -1}  +  \frac{\pi}{4} \sum_{k \neq j}   c_{k, -1} c_{k, 1}  \overline{c}_{j, 1}   \big\}
   \end{cases}
\]
This system is in general quite complicated. We present some of its solutions:
\vskip 0.5cm
\begin{flushleft}
1) $c_{j, 1} = c_{j, 1}^0 \exp \big\{  it (\frac{\pi}{4}  |c_{j, 1}^0|^2 c_{j, 1}^0     ) \big\}  $, $c_{j, -1} =0$ for all $j =1, 2, \cdots , d$.
\vskip 0.5cm
2) $c_{j, -1} = c_{j, -1}^0 \exp \big\{  it (\frac{\pi}{4}  |c_{j, -1}^0|^2 c_{j, -1}^0   ) \big\}  $, $c_{j, 1} =0$ for all $j =1, 2, \cdots , d$.
\vskip 0.5cm
3) For one $j \in \{ 1, 2, \cdots , d \}$, $c_{j, 1} = c_{j, 1}^0 \exp \big\{  it (\frac{\pi}{4}  |c_{j, 1}^0|^2 c_{j, 1}^0 + \frac{\pi}{2} |c_{j, -1}^0|^2 c_{j, 1}^0   ) \big\}  $, $c_{j, -1} = c_{j, -1}^0 \exp \big\{  it (\frac{\pi}{4}  |c_{j, -1}^0|^2 c_{j, -1}^0 + \frac{\pi}{2} |c_{j, 1}^0|^2 c_{j, -1}^0   ) \big\} $, $c_{k, 1} = c_{k, -1} =0$ for all $k \neq j$.
\vskip 0.5cm
4) Let $\overline{c}_{j, -1} = \lambda_j c_{j, 1}$ for all $j =1, 2, \cdots , d$, with each $\lambda_j \in \mathbb{C}$ satisfies $|\lambda_j| =1$. This leads to 
\begin{equation}
\begin{split}
& c_{j, 1} = c_{j, 1}^0 \exp \big\{  it (\frac{\pi}{4}  |c_{j, 1}^0|^2   + \frac{\pi}{2} |c_{j, -1}^0|^2    + \frac{\pi}{4} \sum_{k \neq j}  | c_{k, -1}^0   |^2   +  \frac{\pi}{4} \sum_{k \neq j}  \overline{\lambda_k}  |c_{k, 1}^0   |^2     ) \big\}  ,  \\
& c_{j, -1}  = c_{j, -1}^0 \exp \big\{  it (\frac{\pi}{4}  |c_{j, -1}^0|^2   + \frac{\pi}{2} |c_{j, 1}^0|^2     + \frac{\pi}{4} \sum_{k \neq j} | c_{k, 1}^0   |^2    +  \frac{\pi}{4} \sum_{k \neq j}^0   \overline{\lambda_k}   |c_{k, 1}^0|^2    ) \big\}  .  \\
\end{split}
\end{equation}
\vskip 0.5cm
5) For a pair of $(j, k)$, $j, k \in  \{ 1, 2, \cdots , d \}$, $j \neq k$, let $c_{j, 1} = \lambda c_{k, 1}$, $c_{j, -1} = \mu c_{k, -1}$ with each $\lambda, \mu \in \mathbb{C}$ satisfies $|\lambda| = |\mu| =1$, $\lambda = \overline{\mu}$. Let $c_{l, 1} = c_{l, -1} =0$ for all $l \neq j, k$, $l \in \{ 1, 2, \cdots , d \}$. This leads to 
\begin{equation}
\begin{split}
& c_{j, 1} = c_{j, 1}^0 \exp \big\{  it (\frac{\pi}{4}  |c_{j, 1}^0|^2   + \frac{\pi}{2} |c_{j, -1}^0|^2    + \frac{\pi}{4}   | c_{k, -1}^0   |^2   +  \frac{\pi}{4} \lambda \mu  |c_{k, 1}^0   |^2     ) \big\}  ,  \\
& c_{j, -1}  = c_{j, -1}^0 \exp \big\{  it (\frac{\pi}{4}  |c_{j, -1}^0|^2   + \frac{\pi}{2} |c_{j, 1}^0|^2     + \frac{\pi}{4}  | c_{k, 1}^0   |^2    +  \frac{\pi}{4}    \lambda \mu |c_{k, -1}^0|^2    ) \big\}  .  \\
\end{split}
\end{equation}
\vskip 0.5cm
6) For a pair of $(j, k)$, $j, k \in  \{ 1, 2, \cdots , d \}$, $j \neq k$, let $c_{j, -1} = \lambda c_{k, 1}$, $c_{j, 1} = \mu c_{k, -1}$ with each $\lambda, \mu \in \mathbb{C}$ satisfies $|\lambda| = |\mu| =1$, $\lambda = \overline{\mu}$. Let $c_{l, 1} = c_{l, -1} =0$ for all $l \neq j, k$, $l \in \{ 1, 2, \cdots , d \}$. This leads to 
\begin{equation}
\begin{split}
& c_{j, 1} = c_{j, 1}^0 \exp \big\{  it (\frac{\pi}{4}  |c_{j, 1}^0|^2   + \frac{\pi}{2} |c_{j, -1}^0|^2    + \frac{\pi}{4}   | c_{k, -1}^0   |^2   +  \frac{\pi}{4} \lambda \mu  |c_{k, 1}^0   |^2     ) \big\}  ,  \\
& c_{j, -1}  = c_{j, -1}^0 \exp \big\{  it (\frac{\pi}{4}  |c_{j, -1}^0|^2   + \frac{\pi}{2} |c_{j, 1}^0|^2     + \frac{\pi}{4}  | c_{k, 1}^0   |^2    +  \frac{\pi}{4}    \lambda \mu |c_{k, -1}^0|^2    ) \big\}  .  \\
\end{split}
\end{equation}
\end{flushleft}
All these solutions are quasi-periodic.
Maximizers of $\mathcal{E}$ for fixed mass $M$ are the solutions with $|c_{j, 1}| = |c_{j, -1}|$, and minimizers are the solutions with $c_{j, 1} =0$ or $c_{j, -1} =0$. 
All the solutions can be obtained as extremizers, and are orbitally stable, in the sense that the moduli $|c_{j, 1}(t)|$, $|c_{j, -1}(t)|$ are stable with respect to perturbations of the data, uniformly in time (but the angles are not stable).
 



One can keep continuing to obtain ODE dynamics on each eigenspace $E_N$, $N \geq 2$. We do not pursue this direction here.

\section{Acknowledgement}

The author would like to express her gratitude to Professor Pierre Germain for bringing this research topic to her attention, and also for the very helpful discussions, without which this work would be impossible. 






\end{document}